\documentclass[12pt]{article}

\usepackage{amsmath, graphicx}
\usepackage{amssymb}
\usepackage{amsfonts}
\usepackage{amsthm}
\usepackage{enumerate}
\usepackage{verbatim}
\usepackage{bbm}

\usepackage[bookmarks,bookmarksnumbered,%
    citebordercolor={0.8 0.8 0.8},filebordercolor={0.8 0.8 0.8},%
    linkbordercolor={0.8 0.8 0.8},%
    urlbordercolor={0.8 0.8 0.8},%
    pagebackref,linktocpage%
    ]{hyperref}

\usepackage[margin=1.0in]{geometry}

\allowdisplaybreaks

\newtheorem{theorem}{Theorem}[section]
\newtheorem{proposition}[theorem]{Proposition}
\newtheorem{lemma}[theorem]{Lemma}
\newtheorem{definition}[theorem]{Definition}
\newtheorem{corollary}[theorem]{Corollary}

\newtheorem{remark}{Remark}[section]
\newtheorem{question}{Question}[section]

\theoremstyle{definition}
\newtheorem{eg}{Example}[section]

\numberwithin{equation}{section}
\newcommand{\rr}{{\mathbb R}}
\newcommand{\zz}{{\mathbb Z}}
\newcommand{\nn}{{\mathbb N}}
\newcommand{\cc}{{\mathbb C}}

\newcommand{\bb}[1]{\mathbb {#1}}

\newcommand{\cals}{\mathcal{S}}
\newcommand{\calb}{{\cal B}}
\newcommand{\calm}{{\cal M}}

\newcommand{\diam}{{\rm diam}}

\newcommand{\eps}{\epsilon}
\newcommand{\st}{\,\colon\,}
\newcommand{\R}{\mathbb{R}}
\newcommand{\ind}[1]{\mathbbm{1}_{#1}}

\newcommand{\dist}{\hbox{dist\,}}
\newcommand{\supp}{\mathrm{supp\,}}

\newcommand{\h}[1]{\widehat{#1}}

\let\oldmarginpar\marginpar
\renewcommand\marginpar[1]{\-\oldmarginpar[\raggedleft\footnotesize #1]%
{\raggedright\footnotesize #1}}

\def\keywords#1{\bigskip \par\noindent{\it Keywords and phrases: }#1\par}
\def\AMS#1{\par\noindent{\it 2010 Mathematics Subject Classification: }#1\par}

\title{Finite configurations in sparse sets}
\author{Vincent Chan, Izabella {\L}aba, Malabika Pramanik}
\date{\today}
\begin{document}

\maketitle


\begin{abstract}
Let $E \subseteq \rr^n$ be a closed set of Hausdorff dimension $\alpha$. For $m \geq n$, let $\{B_1,\ldots,B_k\}$ be $n \times (m-n)$ matrices. We prove that if the system of matrices $B_j$ is non-degenerate in a suitable sense, $\alpha$ is sufficiently close to $n$, and if $E$ supports a probability measure obeying appropriate dimensionality and Fourier decay conditions, then for a range of $m$ depending on $n$ and $k$, the set $E$ contains a translate of a non-trivial $k$-point configuration $\{B_1y,\ldots,B_ky\}$. As a consequence, we are able to establish existence of certain geometric configurations in Salem sets (such as parallelograms in $\mathbb R^n$ and isosceles right triangles in $\mathbb R^2$). This can be viewed as a multidimensional analogue of the result of \cite{labapram} on 3-term arithmetic progressions in subsets of $\rr$.
\end{abstract}

\tableofcontents

\noindent \keywords{Finite point configurations, systems of linear equations, Salem sets, Hausdorff dimension, Fourier dimension}
\vskip0.2in

\noindent \AMS{28A78, 42A32, 42A38, 42A45, 11B25}

\section{Introduction}\label{ch:Introduction}

This paper is a contribution to the study of Szemer\'edi-type problems in continuous settings
in Euclidean spaces. Specifically, given a class of subsets of $\rr^n$ that are ``large" in a certain sense,
one may ask whether every set in this class must contain certain geometric configurations.
The precise meaning of this will vary. For instance, given a fixed set $F\subset\rr^n$ (usually discrete, but not necessarily finite), one could ask if
there is a geometrically similar copy of $F$ contained in every set $E$ in a given class; or one
could ask if every such $E$ contains solutions to a given system of
linear equations. This could be viewed as continuous analogues of Szemer\'edi's theorem on arithmetic progressions in sets of integers of positive upper asymptotic density \cite{szek}, or of its multidimensional variants \cite{furst-katz}.

It is an easy consequence of the Lebesgue density theorem that any set $E\subset \rr^n$ of positive Lebesgue measure contains a similar copy of any finite set $F$.
A famous conjecture of Erd{\H{o}}s \cite{erdos} states that for any
infinite set of real numbers $F$, there exists a set $E$ of positive
measure which does not contain any non-trivial affine copy of $F$.
Falconer \cite{falconer} shows using a Cantor-like construction that if $F \subseteq \rr$ contains
a slowly decaying sequence $\{x_n\}$ such that $x_n\searrow 0$ and $\liminf x_{i+1}/x_i = 1$,  there exists a closed set $E \subseteq \rr$ of positive measure which does not contain any affine copies of $F$.
For other classes of negative examples and related results, see \cite{bourg-sequences}, \cite{HL}, \cite{mk}, \cite{komjath}.
For faster decaying sequences, such as the geometric sequence $\{2^{-i}\}$, the question remains open.

We will focus on the case when the configuration is finite, but $E$ has Lebesgue measure 0, making a trivial resolution of the problem impossible. Instead, the assumptions on the size of $E$ will be given in terms of its Hausdorff dimension $\dim_H(E)$.
Recall that by Frostman's lemma, we have
\begin{equation}
\label{frostman}
	\dim_H(E) = \sup\left\{\alpha \in[0,n] \st \exists \mu \in \calm(E) \text{ with }
\sup_{\varepsilon>0}\frac{\mu(B(x,\varepsilon))}{\varepsilon^\alpha}<\infty\right\},
\end{equation}
where $\calm(E)$ is the set of probability measures supported on $E$.

Consider the first non-trivial case where $F$ consists of three points. Then we have the following result in $\rr$.

\begin{theorem}[Keleti, \cite{keleti2}]
\label{thmKeleti2}
For a given distinct triple of points $\{x,y,z\}$, there exists a compact set in $\rr$ with Hausdorff dimension 1 which does not contain any similar copy of $\{x,y,z\}$.
\end{theorem}

Keleti also proves the existence of sets $E\subset\rr$ with $\dim_H(E)=1$ that avoid all ``one-dimensional parallelograms'' $\{x,x+y,x+z,x+y+z\}$, with $y,z \neq 0$ \cite{keleti}, or all similar copies of any 3-point configuration from a given sequence $\{(x_i,y_i,z_i)\}_{i=1}^\infty$ \cite{keleti2}.

In higher dimensions, there is a range or results of this type. For instance, we have the following.

\begin{theorem}[Maga, \cite{maga}]
\label{thmMaga2}
For distinct points $x,y,z \in \rr^2$, there exists a compact set in $\rr^2$ with Hausdorff dimension 2 which does not contain any similar copy of $\{x,y,z\}$.
\end{theorem}

Further examples will be given in Section \ref{intro-examples}.

Additive combinatorics suggests that sets $E$ that are ``pseudorandom" in an appropriate sense should be better behaved with regard to Szemer\'edi-type phenomena than generic sets of the same size. For example, it is well known that there are sets
$A \subset \{1,2,\dots,N\}$ of cardinality at least $C_\varepsilon N^{1-\varepsilon}$ for all $\varepsilon$ such that $A$ does not contain any non-trivial 3-term arithmetic progression (see \cite{salemspencer}, \cite{behrend}).
On the other hand, there are Szemer\'edi-type results available for sets of integers of zero asymptotic density if additional randomness or pseudorandomness conditions are assumed, see e.g. \cite{KLR}, \cite{greentao}.
Different types of  pseudorandomness conditions have been used in the literature, depending on the context and especially on the type of configurations being sought. In the particular case of 3-term arithmetic progressions, the appropriate conditions take a Fourier analytic form \cite{roth}; a more recent (and much deeper) result is that Fourier analytic conditions can also be used to control more general finite configurations of ``true complexity" \cite{gowers-wolf}.

In the continuous setting, this leads us to considering the Fourier dimension of a set $E\subset\rr^n$, defined as
\begin{equation}
\label{fdim}
	\dim_F(E) = \sup\{\beta \in [0,n] \st \exists \mu \in \calm(E) \text{ with }
\sup_{\xi\in\rr^n} |\h{\mu}(\xi)|(1+|\xi|)^{\beta/2} <\infty\}.
\end{equation}

Here, $\h{\mu}(\xi) = \int e^{-2\pi ix \cdot \xi} d\mu(x)$.
It is well-known that $\dim_F(E) \leq \dim_H(E)$ for all $E \subseteq \rr^n$. Strict inequality is possible and common, for instance the middle-thirds Cantor set has Fourier dimension 0 and Hausdorff dimension $\log 2/ \log 3$. When $\dim_F(E)=\dim_H(E)$, we say that $E$ is a \emph{Salem set}. Most of the known constructions of Salem sets are probabilistic, see e.g. \cite{salem}, \cite{kahanetext}, \cite{bluhm-1}, \cite{bluhm-2}, \cite{labapram}; deterministic examples are in \cite{kahane}, \cite{kaufman}.

It turns out that a suitable combination of Hausdorff and Fourier dimensionality conditions does indeed force the presence of three term progressions in subsets of $\rr$.

\begin{theorem}[{\L}aba, Pramanik \cite{labapram}]
\label{thmLabaPram}
Suppose $E \subseteq [0,1]$ is a closed set which supports a probability measure $\mu$ with the following properties:
\begin{enumerate}[(i)]
\item $\mu([x,x+\varepsilon]) \leq C_1\varepsilon^\alpha$ for all $0 < \varepsilon \leq 1$,
\item $|\h{\mu}(\xi)| \leq C_2(1-\alpha)^{-B}|\xi|^{-\beta/2}$ for all $\xi \neq 0$,
\end{enumerate}
where $0 < \alpha < 1$ and $2/3 < \beta \leq 1$. If $\alpha > 1-\varepsilon_0(C_1,C_2,B,\beta)$, then $E$ contains a 3-term arithmetic progression.
\end{theorem}

While Theorem \ref{thmLabaPram} is stated and proved in \cite{labapram} only for arithmetic progressions, the same proof works for any fixed 3-point configuration $\{x,y,z\}$. The assumptions (i), (ii) are not equivalent to the requirement that $E$ be a Salem set, but many constructions of Salem sets produce also a measure that satisfies these conditions; this is discussed in \cite{labapram} in more detail.

The present paper may be seen as an extension of Theorem \ref{thmLabaPram} to $\R^n$, and to a larger class of patterns, by showing that certain point configurations are realized by all sets in $\rr^n$ supporting a measure that satisfies a ball condition of type (\ref{frostman}) and a Fourier decay condition of type (\ref{fdim}).


\subsection{Notation and definitions}

\begin{definition}
\label{nondegenconfig}
Fix integers $n \geq 2$, $k \geq 3$ and $m \geq n$. Suppose $B_1,\ldots,B_k$ are $n \times (m-n)$ matrices.
\begin{enumerate}[(a)]
    \item We say $E$ \emph{contains a} $k$\emph{-point} $\bb{B}$\emph{-configuration} if there exists $x \in \rr^n$ and $y \in \rr^{m-n} \setminus \{0\}$ such that $\{x+B_jy\}_{j=1}^k \subseteq E$.
    \item Given any finite collection of subspaces $V_1,\ldots,V_q \subseteq \rr^{m-n}$ with $\dim(V_i) < m-n$, we say that $E$ \emph{contains a non-trivial} $k$\emph{-point} $\bb{B}$\emph{-configuration with respect to} $(V_1,\ldots,V_q)$ if there exists $x \in \rr^n$ and $y \in \rr^{m-n} \setminus \bigcup_{i=1}^q V_i$ such that $\{x+B_jy\}_{j=1}^k \subseteq E$.
\end{enumerate}
For both of these definitions, we will drop the $k$ from the notation if there is no confusion.
\end{definition}

Let $A_1,\ldots,A_k$ be $n \times m$ matrices. For any set of distinct indices $J = \{j_1,\ldots,j_s\} \subseteq \{1,\ldots,k\}$, define the $ns \times m$ matrix $\bb{A}_J$ by
\[
	\bb{A}_J^t = (A_{j_1}^t \ \cdots \ A_{j_s}^t).
\]
We shall use $\bb{A} = \bb{A}_{\{1,\ldots,k\}}$.

Let $r$ be the unique positive integer such that
\begin{equation}
\label{rdefn}
	n(r-1) < nk-m \leq nr.
\end{equation}
If we have $nk-m$ components and account for $r-1$ groups of size $n$, we are left with $nk-m-n(r-1)$. This quantity is useful in the main theorem, and bulky to constantly use. We shall denote
\begin{equation}
\label{nprimedefn}
	n' = nk-m-n(r-1).
\end{equation}
Notice if $nk-m$ is a multiple of $n$, then $n'=n$, and in general, $0 < n' \leq n$.

\begin{definition}
\label{nondegenmat}
We say that $\{A_1,\ldots,A_k\}$ is \emph{non-degenerate} if for any $J \subseteq \{1,\ldots,k\}$ with $\#(J) = k-r$ and any $j \in \{1,\ldots,k\} \setminus J$, the $m \times m$ matrix
\[
	( \bb{A}_J^t \ \widetilde{A_j}^t )
\]
is non-singular for any choice of $\widetilde{A_j}$ a submatrix of $A_j$. Observe that for the above matrix to be $m \times m$, we require $\widetilde{A_j}$ to consist of $n-n'$ rows.

\end{definition}


\subsection{The main result}

\begin{theorem}
\label{mainresult}
Suppose
\begin{equation}
\label{nmkcond}
	n\left\lceil \frac{k+1}{2} \right\rceil \leq m < nk
\end{equation}
and
\begin{equation}
\label{betacond}
	\frac{2(nk-m)}{k} < \beta < n.
\end{equation}
Let $\{B_1,\ldots,B_k\}$ be a collection of $n \times (m-n)$ matrices such that $A_j = (I_{n \times n} \ B_j)$ is non-degenerate in the sense of Definition \ref{nondegenmat}, where $I_{n \times n}$ is the $n \times n$ identity matrix. Then for any constant $C$, there exists a positive number $\eps_0 = \eps_0(C,n,k,m,\bb{B}) \ll 1$ with the following property. Suppose the set $E \subseteq \rr^n$ with $|E| = 0$ supports a positive, finite, Radon measure $\mu$ with the two conditions:
\begin{enumerate}[(a)]
	\item (ball condition) $\sup_{\substack{x \in E \\ 0 < r < 1}} \frac{\mu(B(x;r))}{r^\alpha} \leq C$ if $n-\varepsilon_0 < \alpha < n$,
	\item (Fourier decay) $\sup_{\xi \in \rr^n} |\h{\mu}(\xi)|(1+|\xi|)^{\beta/2} \leq C$.
\end{enumerate}
Then:
\begin{enumerate}[(i)]
	\item $E$ contains a $k$-point $\bb{B}$-configuration in the sense of Definition \ref{nondegenconfig} (a).
	\item Moreover, for any finite collection of subspaces $V_1,\ldots,V_q \subseteq \rr^{m-n}$ with $\mathrm{dim}(V_i) < m-n$, $E$ contains a non-trivial $k$-point $\bb{B}$-configuration with respect to $(V_1,\ldots,V_q)$ in the sense of Definition \ref{nondegenconfig} (b).
\end{enumerate}
\end{theorem}

Note that $(a)$ implies $E$ has Hausdorff dimension at least $\alpha$ by Frostman's Lemma, and $(b)$ implies that $E$ has Fourier dimension at least $\beta$.

The existence and constructions of measures on $\rr$ that satisfy $(a),(b)$ are discussed in detail in \cite{labapram}. In higher dimensions, it should be possible to generalize the construction in \cite[Section 6]{labapram} to produce examples in $\rr^n$; alternatively, it is easy to check that if $\mu=\tilde\mu(dr)\times \sigma(d\omega)$ is a product measure in radial coordinates $(r,\omega)$, where $\tilde\mu$ is a Salem measure on $[0,1]$ as in \cite[Section 6]{labapram} and $\sigma$ is the Lebesgue measure on $S^{n-1}$, then $\mu$ satisfies the conditions $(a),(b)$ of Theorem \ref{mainresult}.


\subsection{Examples}
\label{intro-examples}

We now give a few examples of geometric configurations covered by Theorem \ref{mainresult}. For proofs and further discussion, see Section \ref{ch:Examples}.

\begin{corollary}\label{cor-triangles}
Let $a,b,c$ be three distinct points in the plane. Suppose that $E\subset\rr^2$ satisfies the assumptions of Theorem \ref{mainresult} with $\varepsilon_0$ small enough depending on $C$ and on the configuration $a,b,c$.
Then $E$ must contain three distinct points $x,y,z$ such that the triangle $\triangle xyz$ is a similar (possibly rotated) copy of the triangle $\triangle abc$.
\end{corollary}

Theorem \ref{thmMaga2} shows that Corollary \ref{cor-triangles} fails without the assumption (b), even if $E$ has Hausdorff dimension 2.

In dimensions $n\geq 2$, one may also consider the following modified question: given a set $E\subset\rr^n$, how large can $\dim_H(E)$ be if $E$ does not contain a triple of points forming a particular angle $\theta$?. For ease, we say $E \subseteq \rr^n$ \emph{contains the angle} $\theta$ if there exist distinct points $x,y,z \in E$ such that the angle between the vectors $y-x$ and $z-x$ is $\theta$, and write $\angle\theta \in E$. Define
\[
	C(n,\theta) = \sup\{s \st \exists E \subseteq \rr^n \text{ compact with } \dim_H(E) = s, \angle \theta \notin E\}.
\]
Harangi, Keleti, Kiss, Maga, M{\'a}th{\'e}, Matilla, and Strenner \cite{harangi} give upper bounds on $C(n,\theta)$ (which they show is tight for $\theta = 0,\pi$), and M{\'a}th{\'e} \cite{mathe2} provides lower bounds. Their results are summarized below.

\bigskip

\begin{tabular}{l|c|c}
$\theta$ & lower bound on $C(n,\theta)$ & upper bound on $C(n,\theta)$ \\
\hline
$0, \pi$ & $n-1$ & $n-1$ \\
$\pi/2$ & $n/2$ & $\lfloor (n+1)/2 \rfloor$ \\
$\cos^2 \theta \in \bb{Q}$ & $n/4$ & $n-1$ \\
other $\theta$ & $n/8$ & $n-1$
\end{tabular}

\bigskip

Corollary \ref{cor-triangles} shows that if $\theta$ is given, then any set $E\subset\rr^2$ as in Theorem \ref{mainresult} must not only contain $\theta$, but in fact $\theta$ can be realized as the angle at the apex of a non-degenerate isosceles triangle with vertices in $E$. This also answers, for such sets, a question posed by
Maga \cite{maga}:

\begin{question}
\label{magaq}
If $E \subseteq \rr^2$ is compact with $\dim_H(E) = 2$, must $E$ contain the vertices of an isosceles triangle?
\end{question}

A different point of view is adopted in \cite{GI}, where the following question is considered. For $E\subset\rr^2$, let $T_2(E)=E^3/\sim$, where $(a,b,c)\sim(a',b',c')$ if and only if the triangles $\triangle abc$ and $\triangle a'b'c'$ are congruent. Observe that $T_2(\rr)$ can be parametrized as a 3-dimensional space, e.g. by one angle and the sidelengths of the two sides adjacent to it. What can we say about the size of $T_2(E)$ if the dimension of $E$ is given?

\begin{theorem}[Greenleaf and Iosevich, \cite{GI}]
\label{gi-thm}
Let $E\subset\rr^2$ be a compact set with $\dim_H(E)>7/4$. Then $T_2(E)$ has positive 3-dimensional measure.
\end{theorem}

Theorem \ref{gi-thm} does not (and, in light of Theorem \ref{thmMaga2}, could not) guarantee the existence of a triangle similar to any particular triangle $\triangle abc$ given in advance. It does, however, ensure that the set of triangles spanned by points of $E$ is large. Our Theorem \ref{mainresult} does not provide this type of results, due to the dependence of $\varepsilon_0$ on $C$ and on the choice of configurations. While it might be possible to keep track of this dependence with more effort, Theorem \ref{gi-thm} is simpler and holds for a much larger class of sets.
For extensions of Theorem \ref{gi-thm} to finite configurations in higher dimensions, see e.g. \cite{EHI}, \cite{GGIP}, \cite{GILP}.

In yet another direction, Bourgain \cite{bourg-simplices} proved that if $E\subset\rr^n$ has positive upper density (with respect to the Lebesgue measure), and if $\Delta$ is a non-degenerate $(k-1)$-dimensional simplex (i.e. a set of $k$ points in general position) with $k\leq n$, then $E$ contains a translated and rotated copy of $\lambda\Delta$ for all $\lambda$ sufficiently large. A discrete analogue of this result was proved more recently by Magyar \cite{magyar}. In dimension 2, Bourgain's result says that there is a $\lambda_0>0$ such that for any $\lambda>\lambda_0$,  $E$ contains two points $x,y$ with $|x-y|=\lambda$. Similarly, if $E\subset\rr^3$ has positive upper density and $a,b,c$ are three distinct and non-collinear points in $\rr^3$, then for all $\lambda$ large enough, $E$ contains translated and rotated copies of $\triangle abc$ rescaled by $\lambda$. However, Bourgain's result does not apply to configurations of $n+1$ or more points in $\rr^n$, such as triangles in $\rr^2$.

A related consequence of Theorem \ref{mainresult} is the following.

\begin{corollary}\label{cor-triples}
Let $a,b,c$ be three distinct colinear points in $\rr^n$. Suppose that $E\subset\rr^n$ satisfies the assumptions of Theorem \ref{mainresult} with $\varepsilon_0$ small enough.
Then $E$ must contain three distinct points $x,y,z$ that form a similar image of the triple $a,b,c$.
\end{corollary}

For $n=2$, this is a special case of Corollary \ref{cor-triples}. While Theorem \ref{mainresult} does not seem to allow for similar images of general triangles in dimensions $n\geq 3$, it does cover the case of 3 colinear points in any dimension. Note in particular that this includes 3-term arithmetic progressions $\{x,x+y,x+2y\}$ with $y\neq 0$, which is a ``degenerate" configuration in the sense of Bourgain \cite{bourg-simplices}.

We now turn to parallelograms in dimensions $2$ and higher.

\begin{corollary}
\label{cor-parallelo}
Suppose that $E\subset\rr^n$ satisfies the assumptions of Theorem \ref{mainresult}, with $\varepsilon_0$ small enough depending on $C$. Then $E$ contains a parallelogram $\{x,x+y,x+z,x+y+z\}$, where the four points are all distinct.
\end{corollary}

Again, this should be compared to a result of Maga, which shows that the result is false without the Fourier decay assumption.

\begin{theorem}[Maga, \cite{maga}]
\label{thmMaga}
There exists a compact set in $\rr^n$ with Hausdorff dimension $n$ which does not contain any parallelogram $\{x,x+y,x+z,x+y+z\}$, with $y,z \neq 0$.
\end{theorem}

We end with a polynomial example.

\begin{corollary}
\label{cor-poly}
Let $a_1,\dots,a_6$ be distinct numbers, all greater than 1.
Suppose that $E\subset\rr^3$ satisfies the assumptions of Theorem \ref{mainresult}, with $\varepsilon_0$ small enough depending on $C$ and $a_i$. Then $E$ contains a configuration of the form
\begin{equation}\label{poly1}
x,\ \  x+B_2y,\ \ x+B_3y,\ \ x+B_4y,
\end{equation}
for some $x\in\rr^3$ and $y\in\rr^6$ with $B_iy\neq 0$ for $i=2,3,4$, where
$$
B_2=\begin{pmatrix} 1&\dots &1 \\ a_1& \dots & a_6 \\ a_1^2 & \dots & a_6^2 \end{pmatrix},\ \
B_3=\begin{pmatrix} a_1^3&\dots & a_6^3 \\ a_1^4& \dots & a_6^4 \\ a_1^5 & \dots & a_6^5 \end{pmatrix},\ \
B_4=\begin{pmatrix} a_1^6&\dots & a_6^6 \\ a_1^7& \dots & a_7^4 \\ a_1^8 & \dots & a_6^8 \end{pmatrix}.\ \
$$
\end{corollary}

This is a non-trivial result in the following sense. Since Vandermonde matrices are non-singular, the set of 6 vectors
$$
\begin{pmatrix} 1 \\ a_1\\ \vdots \\ a_1^5 \end{pmatrix},\ \ \dots, \ \
\begin{pmatrix} 1 \\ a_6\\ \vdots \\ a_6^5 \end{pmatrix}\ \
$$
forms a basis for $\rr^6$. It follows that if $a,b,c\in E$, there is a unique $y\in\rr^6$ such that $b=a+B_2y$ and $c=a+B_3y$. This also determines uniquely the point $a+B_4y$, which might or might not be in $E$. Our result asserts that, under the conditions of Corollary \ref{cor-poly}, we may choose $x$ and $y$ so that in fact all 4 points in (\ref{poly1}) lie in $E$.

\subsection{Outline of proof}

We introduce the following multilinear form:
\begin{equation}
\label{lambdadefn-intro}
	\Lambda(f_1,\ldots,f_k) = \int_{\rr^m} \prod_{j=1}^k f_j(A_j\vec{x})\ d\vec{x}.
\end{equation}
If $f_j = f$ for all $j$, we write $\Lambda(f)$ instead of $\Lambda(f,\ldots,f)$.
The use of multilinear forms formally similar to (\ref{lambdadefn-intro}) is common in the literature on Szemer\'edi-type problems.

Our strategy, roughly following that of \cite{labapram}, is to define an analogue of (\ref{lambdadefn-intro}) for measures via its Fourier-analytic representation (\ref{fourierlambda-intro}) below, prove that this analogue can be used to count the $k$-point configurations we seek, and obtain lower bounds on it that imply the existence of such configurations. A key feature in the present work is the multidimensional geometry of the problem, determined by the system of matrices $A_j$.  While this issue is almost nonexistent in \cite{labapram}, here it will play a major role at every stage of the proof and will account for most of the difficulties, both technical and conceptual.

The multilinear form (\ref{lambdadefn-intro}) is initially defined for $f_j \in C_c^\infty(\rr^n)$.
We prove in Proposition \ref{fourierform} that for such functions we have
\begin{equation}
\label{fourierlambda-intro}
	\Lambda(f_1,\ldots,f_k) = C\int_S \prod_{j=1}^k \widehat{f_j}(\xi_j)\ d\sigma(\xi_1, \cdots ,\xi_k),
\end{equation}
where $\sigma$ is the Lebesgue measure on the subspace
\begin{equation}
\label{sdefn-intro}
	S = \left\{\xi = (\xi_1,\ldots,\xi_k) \in (\rr^n)^k \st \sum_{j=1}^k A_j^t\xi_j = \vec{0}\right\}
\end{equation}
and $C = C(\bb{A})$ is a constant only depending on the matrices $A_j$. The importance of (\ref{fourierlambda-intro}) for us is twofold. First, it allows us to use Fourier bounds on $f_j$ to control the size of $\Lambda(f_1,\dots,f_k)$, a fact that we will use repeatedly in the paper. Second, unlike (\ref{lambdadefn-intro}), (\ref{fourierlambda-intro}) makes sense for more general (possibly singular) measures $\mu_j$ instead of $C_c^\infty$ functions, provided that their Fourier transforms are well enough behaved so that the integral
\begin{equation}
\label{fouriermu-intro}
	\Lambda^*(\widehat{\mu_1},\ldots,\widehat{\mu_k}) = C\int_S \prod_{j=1}^k \widehat{\mu_j}(\xi_j)\ d\sigma(\xi_1, \cdots ,\xi_k),
\end{equation}
converges.

We prove in Proposition \ref{lambdaextbound} that the integral in (\ref{fouriermu-intro}) does indeed converge, provided that $\mu_j$ and $A_j$ obey the assumptions of Theorem \ref{mainresult}.
While the full strength of these assumptions is not needed at this point, we would like to emphasize two key conditions. First, we require pointwise decay of $\widehat{\mu_j}$. Second, in the absence of additional assumptions on the matrices $A_j$, the decay of $\widehat{\mu_j}(\xi_j)$  in the $\xi_j$ variables would not necessarily translate into decay in any direction on the subspace $S$ in (\ref{sdefn-intro}).
The nondegeneracy conditions in Definition \ref{nondegenmat} ensure that $S$ is in ``general position" relative to the subspaces $\{\xi_j=0\}$ of the full configuration space $(\rr^n)^k$ along which the functions $\widehat{\mu_j}(\xi_j)$ do not decay. This puts us in the best possible geometric case with regard to the convergence of (\ref{fouriermu-intro}), and allows us to complete the proof of the proposition. Similar geometric issues (with somewhat different details) arise later in the proof as well, notably in the proofs of Propositions  \ref{multlinmeas} and \ref{absctslowerbd}.

Let $\Lambda^*(\h{\mu})$ be the multilinear form thus defined, with $\mu_1=\dots=\mu_k=\mu$. We claim that a lower bound of the form
\begin{equation}
\label{lower-intro}
\Lambda^*(\h{\mu})>0
\end{equation}
implies the existence of the $k$-point configurations we seek. If $\mu$ were absolutely continuous with density $f$, this would follow trivially from Proposition \ref{fourierform}, since (\ref{lower-intro}) would be equivalent to a bound $\Lambda(f)>0$, and $\Lambda(f)$ has a direct interpretation in terms of such configurations. For singular measures, however, (\ref{lambdadefn-intro}) need not make sense.

We therefore proceed less directly, following the same route as in \cite{labapram}.
Namely, we prove in Proposition \ref{multlinmeas} that
there exists a non-negative, finite, Radon measure $\nu = \nu(\mu)$ on $[0,1]^{m}$ such that
\begin{itemize}
	\item $\nu(\rr^{m}) = \Lambda^*(\h{\mu})$.
	\item $\supp \nu \subseteq \{x \in \rr^{m} \st A_1x,\ldots,A_kx \in \supp \mu\}$.
	\item For any subspace $V \subseteq \rr^m$ with $\dim V < m$, $\nu(V) = 0$.
\end{itemize}
The last condition implies that the $\nu$-measure of the set of ``degenerate" configurations is 0.
It follows that (\ref{lower-intro}) indeed implies the existence of desired configurations in $\supp\mu$.

It remains to prove (\ref{lower-intro}). Following the strategy of \cite{labapram} again, we decompose $\mu$ as $\mu=\mu_1+\mu_2$, where $\mu_1$ is absolutely continuous with bounded density, and $\mu_2$ is singular but obeys good Fourier bounds.
We then have
\begin{equation}\label{decomp-intro}
	\Lambda^*(\widehat{\mu}) = \Lambda^*(\widehat{\mu_1}) + \Lambda(\widehat{\mu_2},\widehat{\mu_1},\ldots,\widehat{\mu_1}) + \cdots + \Lambda(\widehat{\mu_2}).
\end{equation}
We will treat the first term on the right side of (\ref{decomp-intro}) as the main term, and the remaining terms as error terms.

To bound $\Lambda(\mu_1)$ from below, we need a quantitative Szemer\'edi-type estimate from below for bounded functions.
In Proposition \ref{absctslowerbd}, we prove a bound
\begin{equation}\label{lower-ac-intro}
\Lambda(f) \geq c(\delta,M)
\end{equation}
for all functions $f: [0,1]^n \to \rr$ such that $0 \leq f \leq M$ and $\int f \geq \delta$. To this end, we will modify the ``quantitative ergodic" proof of Varnavides' Theorem given in \cite{tao}. While in \cite{labapram} this proof could be reused almost verbatim, more substantial changes are needed in the multidimensional case.
This will take up the bulk of Section \ref{ch:AbsContEstimates}.

The proof of Theorem \ref{mainresult} is completed in Section \ref{ch:ProofOfMainResult}. We first carry out the decomposition $\mu=\mu_1+\mu_2$ as described earlier, and prove the required bounds on the density of $\mu_1$ and the Fourier transform of $\mu_2$. The first term on the right side of (\ref{decomp-intro}) is bounded from below by a constant $c>0$ using (\ref{lower-ac-intro}), and the sum of the remaining terms can be shown to be less than $c$ in absolute value, using Proposition \ref{lambdaextbound} again and the Fourier estimates on $\mu_2$.

The assumptions of Theorem \ref{mainresult} are sufficient for the entire proof to go through. However, many of our intermediate results hold under weaker conditions. We indicate this explicitly in the statements of the results in question, for clarity and possible use in future work.

\subsection{Acknowledgement}

The second and third authors were supported by NSERC Discovery Grants.


\section{A functional multilinear form for counting configurations}\label{ch:AFunctionalMultiLinearForm}

We will consider the following multilinear form, initially defined for $f_j \in C_c^\infty(\rr^n)$:
\begin{equation}
\label{lambdadefn}
	\Lambda(f_1,\ldots,f_k) = \int_{\rr^m} \prod_{j=1}^k f_j(A_j\vec{x})\ d\vec{x}.
\end{equation}
If $f_j = f$ for all $j$, we write $\Lambda(f)$ instead of
$\Lambda(f,\ldots,f)$. Clearly if $\Lambda(f) \ne 0$, then the support
of $f$ contains configurations of the form $\{ A_j \vec{x} : 1 \leq j
\leq k\}$ for a set of $\vec{x}$ of positive measure.  We will rewrite $\Lambda(f_1,\ldots,f_k)$ in a form that will allow us to extend it to measures, not just functions.

\begin{proposition}
\label{fourierform}
For $f_j \in C_c^\infty(\rr^n)$, $\Lambda(f_1,\ldots,f_k)$ defined in (\ref{lambdadefn}) admits the representation
\[
	\Lambda(f_1,\ldots,f_k) = C\int_S \prod_{j=1}^k \widehat{f_j}(\xi_j)\ d\sigma(\xi_1, \cdots ,\xi_k),
\]
where $\sigma$ is the Lebesgue measure on the subspace
\begin{equation}
\label{sdefn}
	S = \left\{\xi = (\xi_1,\ldots,\xi_k) \in (\rr^n)^k \st \sum_{j=1}^k A_j^t\xi_j = \vec{0}\right\}
\end{equation}
and $C = C(\bb{A})$ is a constant only depending on the matrices $A_j$.
\end{proposition}

\begin{proof}
For $\Phi \in \cals(\rr^m)$ with $\Phi(0) = 1$, the dominated convergence theorem gives
\[
	\Lambda(f_1,\ldots,f_k) = \lim_{\varepsilon \to 0^+} \int_{\rr^m} \left(\prod_{j=1}^k f_j(A_j\vec{x})\right) \Phi(\vec{x}\varepsilon)\ d\vec{x}.
\]
Applying the Fourier inversion formula in $\rr^n$, $g(y) = \int_{\rr^n} e^{2\pi iy \cdot \xi} \h{g}(\xi)\ d\xi$ where $\h{g}(\xi) = \int_{\rr^n} e^{-2\pi ix \cdot \xi}g(x)\ dx$, we get
\begin{align*}
	\Lambda(f_1,\ldots,f_k) & = \lim_{\varepsilon \to 0^+} \int_{\rr^m} \prod_{j=1}^k \left[ \int_{\rr^n} e^{2\pi iA_j\vec{x} \cdot \xi_j}\h{f}_j(\xi_j)\ d\xi_j\right] \Phi(\vec{x}\varepsilon)\ d\vec{x} \\
			& = \lim_{\varepsilon \to 0^+} \int_{\vec{\xi}=(\xi_1,\ldots,\xi_k) \in (\rr^n)^k} \prod_{j=1}^k \h{f}_j(\xi_j) \left[ \int_{\rr^m} e^{2\pi i\vec{x} \cdot \bb{A}^t\xi}\ \Phi(\vec{x}\varepsilon)\ d\vec{x}\right]\ d\vec{\xi} \\
			& = \lim_{\varepsilon \to 0^+} \int_{\vec{\xi} \in (\rr^n)^k} \prod_{j=1}^k \h{f}_j(\xi_j) \frac{1}{\varepsilon^{p}}\h{\Phi}\left(\frac{\bb{A}^t\xi}{\varepsilon}\right)\ d\vec{\xi} \\
			& = C_{\bb{A}^t} \int_S \prod_{j=1}^k \h{f}_j(\xi_j)\ d\sigma(\xi),
\end{align*}
for some constant $C_{\bb{A}^t}$. This last step follows from Proposition \ref{approxidenprop}, with $p = m$, $d = nk$, $V = S$, $P = \bb{A}^t$, and $F = \prod \h{f}_j$.
\end{proof}

\begin{proposition}
\label{fourierform2}
Let $g \in \cals(\rr^m)$, $f_j \in C_c^\infty(\rr^n)$. Then the integral
\[
	\Theta(g;f_1,\ldots,f_k) := \int_{\rr^m} g(\vec{x}) \prod_{j=1}^k f_j(A_j\vec{x})\ d\vec{x}
\]
is absolutely convergent, and admits the representation
\[
	\Theta(g;f_1,\ldots,f_k) = \int_{(\rr^n)^k} \h{g}(-\bb{A}^t\vec{\xi})\prod_{j=1}^k \h{f}_j(\vec{\xi}_j)\ d\vec{\xi}.
\]
\end{proposition}

\begin{proof}
By Fourier inversion,
\[
	\Theta(g;f_1,\ldots,f_k) = \int_{\rr^m} g(\vec{x}) \prod_{j=1}^k \left[\int_{\rr^n} e^{2\pi iA_j\vec{x} \cdot \vec{\xi}_j} \h{f}_j(\vec{\xi}_j)\ d\vec{\xi}_j\right]\ d\vec{x},
\]
which is absolutely convergent since $\h{f}_j,g \in \cals(\rr^m)$. Then by Fubini's Theorem,
\begin{align*}
	\Theta(g;f_1,\ldots,f_k) & = \int_{(\rr^n)^k} \prod_{j=1}^k \h{f}_j(\vec{\xi}_j)\left[\int_{\rr^m} g(x)e^{2\pi i\vec{x} \cdot \bb{A}^t\vec{\xi}}\ d\vec{x}\right]\ d\vec{\xi} \\
			& = \int_{(\rr^n)^k} \h{g}(-\bb{A}^t\vec{\xi})\prod_{j=1}^k \h{f}_j(\vec{\xi}_j)\ d\vec{\xi},
\end{align*}
where the last line follows by the definition of the Fourier transform.
\end{proof}



\section{Extending the multilinear form to measures}\label{ch:ExtensionsToMeasures}

With $S$ as in (\ref{sdefn}), denote $S^\perp = \{\tau \in (\rr^n)^k \st \langle \tau,\xi \rangle = 0 \text{ for all } \xi \in S\}$ and fix a $\tau \in S^\perp$. We will use the variable
\begin{equation}
\label{vardefn}
	\eta = (\eta_1,\ldots,\eta_k) = \xi + \tau = (\xi_1,\ldots,\xi_k) + (\tau_1,\ldots,\tau_k) \in S + \tau,
\end{equation}
where $\xi \in S$, and $\eta_j, \xi_j, \tau_j \in \rr^n$. Define
\begin{equation}
\label{lambdastar}
	\Lambda^*_\tau(g_1,\ldots,g_k) = \int_{S+\tau} \prod_{j=1}^k g_j(\eta_j)\ d\sigma(\eta),
\end{equation}
initially defined for $g_j \in C_c(\rr^n)$ so that the integral is absolutely convergent. If $g_j = g$ for all $j$, we write $\Lambda^*_\tau(g)$ instead of $\Lambda^*_\tau(g,\ldots,g)$. We will use $\Lambda^*$ in place of $\Lambda^*_0$. Our next proposition shows that we may extend this multilinear form to continuous functions with appropriate decay.

In applications, $g_j$ will be $\h{\mu}$, with $\mu$ as in Theorem \ref{mainresult}. While defining the multilinear form $\Lambda$ for measures requires only the use of $\Lambda^*_0$ (so that the integration is on $S$ as in Proposition \ref{fourierform}), the proof of our main result will rely crucially on estimates on $\Lambda^*_\tau$ uniform in $\tau$.

\begin{proposition}
\label{lambdaextbound}
Let $\{A_1,\ldots,A_k\}$ be a non-degenerate collection of $n \times m$ matrices in the sense of Definition \ref{nondegenmat}. Assume that $nk/2 < m < nk$ and $g_1,\ldots,g_k: \rr^n \to \cc$ are continuous functions satisfying
\begin{equation}
	|g_j(\kappa)| \leq M(1+|\kappa|)^{-\beta/2}, \hspace{1cm} \kappa \in \rr^n,
\end{equation}
for some $\beta > 2(nk-m)/k$. Then the integral defining $\Lambda^*_\tau = \Lambda^*_\tau(g_1,\ldots,g_k)$ is absolutely convergent for every $\tau \in S^\perp$. Indeed,
\[
	\sup_{\tau \in S^\perp} \Lambda^*_\tau(|g_1|,\ldots,|g_k|) \leq C
\]
where $C$ depends only on $n,k,m,M$, and $\bb{A}$.
\end{proposition}

Since $S+\tau$ for $\tau \in S^\perp$ includes all possible translates of $S$, Proposition \ref{lambdaextbound} gives a uniform upper bound on the integral of $\prod g_j$ over all affine copies of $S$.

The proof of this proposition is based on a lemma which requires some additional notation. Let $0_{a \times b}$ denote the $a \times b$ matrix consisting of $0$'s; we will also use $0$ when the size is evident. Let $I_{n \times n}$ denote the $n \times n$ identity matrix; we will also use $I_n$ for short if the context is clear. Define the $nk \times n$ matrix $E_j$ ($1 \leq j \leq k$) by
\[
	E_j^t = (0_{n \times n} \ \cdots \ I_{n \times n} \ \cdots \ 0_{n \times n}),
\]
where $I_{n \times n}$ is in the $j$th block. For any $\vec{\varepsilon} = (\varepsilon_1,\ldots,\varepsilon_n) \in \{0,1\}^n$, we denote $I_{n \times n}(\vec{\varepsilon}) = \mathrm{diag}(\varepsilon_1,\ldots,\varepsilon_n)$. For a subset $J' \subseteq \{1,\ldots,n\}$ and index $j \in \{1,\ldots,k\}$, define the $nk \times n$ matrix
\[
	E_j(J')^t = (0_{n \times n} \ \cdots \ I_{n \times n}(\varepsilon) \ \cdots \ 0_{n \times n}),
\]
where $I_{n \times n}(\varepsilon)$ is in the $j$th block and $\varepsilon_i = 1$ if and only if $i \in J'$. Finally, for $\xi \in S$ we define
\[
	\xi_j(J') = E_j(J')\xi.
\]

\begin{lemma}
\label{lembasis}
Let $\{A_1,\ldots,A_k\}$ be a non-degenerate collection of $n \times m$ matrices in the sense of Definition \ref{nondegenmat}. Let $J = \{j_1,\ldots,j_r\} \subseteq \{1,\ldots,k\}$ and $J' \subseteq \{1,\ldots,n\}$ be collections of distinct indices with $\#(J') = n'$, defined by (\ref{nprimedefn}). Then the projection of $(\xi_{j_1},\ldots,\xi_{j_{r-1}},\xi_{j_r}(J'))$ on $S$ is a coordinate system on $S$  as defined in (\ref{sdefn}). In particular, there exists a constant $C = C(J,J',j,\bb{A})$ such that $d\sigma(\xi) = C\ d\xi_{j_1} \cdots d\xi_{j_{r-1}} d\xi_{j_r}(J')$, where $\sigma(\xi)$ is the Lebesgue measure on $S$.
\end{lemma}

\begin{proof}
It suffices to prove \begin{equation}
\label{gencond1}
	S' = \{\xi \in S \st \xi_{j_1} = \xi_{j_2} = \cdots = \xi_{j_{r-1}} = \xi_{j_r}(J') = 0\} \Longrightarrow \dim S' = 0.
\end{equation}

We will examine $S' = \{\xi \in S \st \xi_1 = \xi_2 = \cdots = \xi_{r-1} = \xi_r(J') = 0\}$; the other cases are similar. $S'$ is the subspace defined by
\[
	\left\{\xi \in (\rr^n)^k \st \begin{pmatrix}
  A_1^t & A_2^t & \cdots & A_{r-1}^t & A_r^t & A_{r+1}^t & \cdots & A_k^t \\
  I_n & 0 & \cdots & 0 & 0 & 0 & \cdots & 0 \\
  0 & I_n & \cdots & 0 & 0 & 0 & \cdots & 0 \\
  \vdots & \vdots & \ddots & \vdots & \vdots & \vdots & & \vdots \\
  0 & 0 & \cdots & I_n & 0 & 0 & \cdots & 0 \\
  0 & 0 & \cdots & 0 & I_n(\vec{1}-\vec{\varepsilon}) & 0 & \cdots & 0
 \end{pmatrix}\xi = 0\right\},
\]
where $\varepsilon_i = 1$ if and only if $i \in J'$, and $\vec{1}=(1,\dots,1)$. For $\dim S' = 0$, we need the kernel of the above $(m+nr) \times nk$ matrix to have dimension 0, so we need the rank of said matrix to be $nk$. Notice $I_n$ is of rank $n$ and we have $r-1$ of these, and $I(\vec{1}-\vec{\varepsilon})$ has rank $[nk-m-n(r-1)]$, so it suffices for
\[
	\mathrm{rank}\begin{pmatrix}
  (A_r(\vec{\varepsilon}))^t & A_{r+1}^t & \cdots & A_k^t \end{pmatrix} = nk-nr.
\]
A similar condition holds if we examine $S' = \{\xi \in S \st \xi_{i_1} = \xi_{i_2} = \cdots = \xi_{i_{r-1}} = \xi_{i_r}(J') = 0\}$ in general, and upon taking the transpose we arrive at the sufficient condition
\begin{equation}
\label{gencond2alt}
    \mathrm{rank}\begin{pmatrix}
  \bb{A}_I \\
  A_{i_{k-(r-1)}}(\vec{\varepsilon}) \end{pmatrix} = m
\end{equation}
for $I = \{i_1,\ldots,i_{k-r}\} \subseteq \{1,\ldots,k\}$ a set of distinct indices, and $\vec{\varepsilon} \cdot \vec{1} = n-n'$. Notice this means the matrix is of full rank. (\ref{gencond2alt}) follows from the non-degeneracy assumption in Definition \ref{nondegenmat}.
\end{proof}

\begin{proof}[Proof of Proposition 3.1]
Fix $\tau \in S^\perp$. We use $\mathrm{Sym}_k$ to denote the symmetric group on $k$ elements. For a permutation $\theta \in \mathrm{Sym}_k$, we define the region
\[
	\Omega_\theta = \{\eta \in S+\tau \st |\eta_{\theta(1)}| \leq |\eta_{\theta(2)}| \leq \cdots \leq |\eta_{\theta(k)}|\}
\]
so that
\[
	S+\tau = \bigcup_{\theta \in \mathrm{Sym}_k} \Omega_\theta.
\]
For simplicity and without loss of generality, we will examine the case of $\theta = \mathrm{id}$ and write $\Omega_\mathrm{id}$ as $\Omega$; the other cases are analogous. It then suffices to show the convergence of the integral
\begin{align*}
	I & = \int_{\Omega} \prod_{j=1}^k |g_j(\eta_j)|\ d\sigma \lesssim \int_{\Omega} \prod_{j=1}^k (1+|\eta_j|)^{-\beta/2}\ d\sigma.
\end{align*}

Let $L = \prod_{j=1}^{r-1} (1+|\eta_j|)^{-\beta/2}$. By H\"{o}lder's inequality over $k-(r-1)$ terms,
\begin{align*}
	I & \lesssim \int_{\Omega} \prod_{j=1}^k (1+|\eta_j|)^{-\beta/2}\ d\sigma \\
			& = \int_{\Omega} L\prod_{i=r}^k (1+|\eta_i|)^{-\beta/2}\ d\sigma \\
			& = \int_{\Omega} \prod_{i=r}^k \left[L^{1/(k-(r-1))}(1+|\eta_i|)^{-\beta/2}\right]\ d\sigma \\
			& \leq \prod_{i=r}^k \left[\int_{\substack{S+\tau \\ |\eta_1| \leq \cdots \leq |\eta_{r-1}| \leq |\eta_i|}} L(1+|\eta_i|)^{-\tfrac{\beta}{2}(k-(r-1))}\ d\sigma\right]^{1/(k-(r-1))} \\
			& \leq \int_{\substack{S+\tau \\ |\eta_1| \leq \cdots \leq |\eta_r|}} \prod_{j=1}^{r-1} (1+|\eta_j|)^{-\beta/2}(1+|\eta_r|)^{-\tfrac{\beta}{2}(k-(r-1))}\ d\sigma.
\end{align*}
In the last inequality, we see that for each $i$, the integral is the same with a different dummy variable, so we collect the terms under the single index $i = r$. Recalling our decomposition of $\eta$ into $\xi$ and $\tau$ as per (\ref{vardefn}) and since $\tau$ is fixed, we arrive at
\[
	I \lesssim \int_{\substack{S \\ |\xi_1+\tau_1| \leq \cdots \leq |\xi_r+\tau_r|}} \prod_{j=1}^{r-1} (1+|\xi_j+\tau_j|)^{-\beta/2}(1+|\xi_r+\tau_r|)^{-\tfrac{\beta}{2}(k-(r-1))}\ d\sigma(\xi).
\]

Suppose $\eta_r = (\eta_{r,1},\ldots,\eta_{r,n})$, and for any permutation $\pi \in \mathrm{Sym}_n$, let
\[
	\eta_r^\pi = (\eta_{r,\pi(1)},\ldots,\eta_{r,\pi(n')}).
\]
As in the beginning of this section, we may partition our current region of integration into a finite number of regions of the form $\{|\eta_{r,\pi(1)}| \geq \cdots \geq |\eta_{r,\pi(n)}|\}$ for $\pi \in \mathrm{Sym}_n$. Then $|\eta_r| \leq n|\eta_{r,\pi(1)}|$, so that $|\eta_r| \sim |\eta_r^\pi|$.

By Lemma \ref{lembasis}, $d\sigma(\xi) = C\ d\xi_1 \cdots d\xi_{r-1} d\xi_r^\pi$, hence
\begin{align*}
	I & \lesssim \int_{\substack{S \\ |\xi_1+\tau_1| \leq \cdots \leq |\xi_r+\tau_r|}} \prod_{j=1}^{r-1} (1+|\xi_j+\tau_j|)^{-\beta/2}(1+|\xi_r^\pi+\tau_r^\pi|)^{-\tfrac{\beta}{2}(k-(r-1))}\ d\xi_1\cdots d\xi_{r-1}d\xi_r^\pi \\
			& \lesssim \int_{\rr^{n'}} \prod_{j=1}^{r-1} \left[\int_{\substack{\xi_j+\tau_j \in \rr^n \\ |\xi_j+\tau_j| \leq |\xi_r^\pi+\tau_r^\pi|}} (1+|\xi_j+\tau_j|)^{-\beta/2}\ d\xi_j\right](1+|\xi_r^\pi+\tau_r^\pi|)^{-\tfrac{\beta}{2}(k-(r-1))}\ d\xi_r^\pi.
\end{align*}
Translating $\xi_j$ by $\tau_j$,
\begin{align*}
	I & \lesssim \int_{\rr^{n'}} \prod_{j=1}^{r-1} \left[\int_{\substack{\xi_j \in \rr^n \\ |\xi_j| \leq |\xi_r^\pi|}} (1+|\xi_j|)^{-\beta/2}\ d\xi_j\right](1+|\xi_r^\pi|)^{-\tfrac{\beta}{2}(k-(r-1))}\ d\xi_r^\pi \\
			& \lesssim \int_{\rr^{n'}} \left[\int_0^{|\xi_r^\pi|} (1+\rho)^{-\beta/2+n-1}\ d\rho\right]^{r-1}(1+|\xi_r^\pi|)^{-\tfrac{\beta}{2}(k-(r-1))}\ d\xi_r^\pi \\
			& \lesssim \int_{\rr^{n'}} (1+|\xi_r^\pi|)^{(n-\beta/2)(r-1)-\tfrac{\beta}{2}(k-(r-1))}\ d\xi_r^\pi \\
			& = \int_{\rr^{n'}} (1+|\xi_r^\pi|)^{n(r-1)-\beta k/2}\ d\xi_r^\pi,
\end{align*}
where the Jacobian in making the spherical change of coordinates $\rho = |\xi_j|$ above is independent of $\tau_j$. This last expression is finite (with a bound independent of $\tau$) when
\[
	\beta k/2-n(r-1) > n' = nk-m-n(r-1),
\]
which holds since $2(nk-m)/k < \beta < n$ and $m > nk/2$.
\end{proof}



\section{Counting geometric configurations in sparse sets}\label{ch:AMultiLinearFormForMeasures}

In this section, we will show that the multilinear form $\Lambda^*$ defined in (\ref{lambdastar}) is effective in counting non-trivial configurations supported on appropriate sparse sets.

\begin{proposition}
\label{multlinmeas}
Suppose $nk/2 < m < nk$ and $2(nk-m)/k < \beta < n$. Let $\{A_1,\ldots,A_k\}$ be a collection of $n \times m$ matrices that are non-degenerate in the sense of Definition \ref{nondegenmat}. Let $\mu$ be a positive, finite, Radon measure $\mu$ with
\begin{equation}
\label{mudecay}
	\sup_{\xi \in \rr^n} |\h{\mu}(\xi)|(1+|\xi|)^{\beta/2} \leq C.
\end{equation}
Then there exists a non-negative, finite, Radon measure $\nu = \nu(\mu)$ on $[0,1]^{m}$ such that
\begin{enumerate}[(a)]
	\item $\nu(\rr^{m}) = \Lambda^*(\h{\mu})$.
	\item $\supp \nu \subseteq \{x \in \rr^{m} \st A_1x,\ldots,A_kx \in \supp \mu\}$.
	\item For any subspace $V \subseteq \rr^m$ with $\dim V < m$, $\nu(V) = 0$.
\end{enumerate}
\end{proposition}


\subsection{Existence of candidate {$\nu$}}

Fix a non-negative $\phi \in \cals(\rr^m)$ with $\int \phi = 1$ and let $\phi_\varepsilon(y) = \varepsilon^{-n}\phi(\varepsilon^{-1}y)$. Let $\mu_\varepsilon = \mu*\phi_\varepsilon$. Notice $\h{\phi} \in \cals(\rr^m)$ since $\phi \in \cals(\rr^m)$, so
\begin{equation}
\label{muepsilondecay}
	|\h{\mu}_\varepsilon(\xi)| = |\h{\mu}(\xi)\h{\phi}(\varepsilon\xi)| \leq C(1+|\xi|)^{-\beta/2}
\end{equation}
with $C = \|\h{\phi}\|_\infty$ independent of $\varepsilon$. Furthermore, $\h{\phi}(\varepsilon \xi) \to \h{\phi}(0) = \int \phi = 1$ as $\varepsilon \to 0$, hence
\begin{equation}
\label{muepsilontomu}
	\h{\mu}_\varepsilon(\xi) \to \h{\mu}(\xi) \text{ pointwise as } \varepsilon \to 0.
\end{equation}
We prove that the multilinear form $\Lambda^*_\tau$ satisfies a weak continuity property, in the following sense:

\begin{lemma}
\label{lambdamuepstolambdamu}
$\Lambda^*_\tau(\h{\mu}_\varepsilon) \to \Lambda^*_\tau(\h{\mu})$ as $\varepsilon \to 0$ for every fixed $\tau \in S^\perp$.
\end{lemma}

\begin{proof}
Fix $\tau \in S^\perp$. By definition,
\[
	\Lambda^*_\tau(\h{\mu}_\varepsilon) = \int_{S+\tau} \prod_{j=1}^k \h{\mu}_\varepsilon(\xi_j)\ d\xi.
\]
By (\ref{muepsilondecay}), $|\h{\mu}_\varepsilon(\eta)| \leq C(1+|\eta|)^{-\beta/2} =: g(\eta)$ uniformly in $\varepsilon$, so
\[
	\prod_{j=1}^k |\h{\mu}_\varepsilon(\xi_j)| \leq C\prod_{j=1}^k g(\xi_j).
\]
By Proposition \ref{lambdaextbound}, $\Lambda^*_\tau(g)$ is finite, so by (\ref{muepsilontomu}) and the dominated convergence theorem,
\[
	\Lambda^*_\tau(\h{\mu}_\varepsilon) \to \Lambda^*_\tau(\h{\mu}).
\]
\end{proof}

For $F \in C([0,1]^m)$ such that $\h{F} \in \cals(\rr^m)$, define the linear functional $\nu$ by
\begin{equation}
\label{nudefn}
	\langle \nu, F \rangle = \lim_{\varepsilon \to 0} \int_{\rr^m} F(\vec{x})\prod_{j=1}^k \mu_\varepsilon(A_j\vec{x})\ d\vec{x}.
\end{equation}
We will prove in Lemma \ref{nulimexistsandbdd} below that the limit exists and extends as a bounded linear functional on $C([0,1])$. Clearly, $\langle \nu, F \rangle \geq 0$ if $F \geq 0$. By the Riesz representation theorem, there exists a non-negative, finite, Radon measure $\nu$ that identifies this linear functional; namely $\langle \nu, F \rangle = \int F \ d\nu$.

\begin{lemma}
\label{nulimexistsandbdd}
There exists a non-negative, bounded, linear functional $\nu$ on $C([0,1])$, that is,
\begin{equation}
\label{nubound}
	|\langle \nu, F \rangle| \leq C\|F\|_\infty
\end{equation}
for some positive constant $C$ independent of $F \in C([0,1])$, which agrees with (\ref{nudefn}) if $\h{F} \in \cals(\rr^m)$.
\end{lemma}

\begin{proof}
Assume the limit (\ref{nudefn}) exists for $F \in C([0,1])$, $\h{F} \in \cals(\rr^m)$. Then
\begin{align*}
    |\langle \nu, F \rangle| & \leq \lim_{\varepsilon \to 0} \int_{\rr^m} |F(\vec{x})|\prod_{j=1}^k \mu_\varepsilon(A_j\vec{x})\ d\vec{x} \\
        & \leq \|F\|_\infty \lim_{\varepsilon \to 0} \Lambda(\mu_\varepsilon) \\
        & = \|F\|_\infty \lim_{\varepsilon \to 0} \Lambda^*(\h{\mu}_\varepsilon) \\
        & \leq C\|F\|_\infty,
\end{align*}
where the last line follows by Proposition \ref{lambdaextbound}, with a constant $C$ independent of $\varepsilon$. Thus, (\ref{nubound}) holds.

It remains to prove that $\langle \nu, F \rangle$ is well-defined. We will prove this by showing that the limit in (\ref{nudefn}) exists for $F \in C([0,1]^m)$ such that $\h{F} \in \cals(\rr^m)$ and use density arguments to extend the functional to all of $C([0,1]^m)$. Applying Corollary \ref{fourierform2} with $g = F$, $f_1 = \ldots = f_k = \mu_\varepsilon$, we obtain
\begin{equation}
\label{nufourierform}
	\langle \nu, F \rangle = \lim_{\varepsilon \to 0} \Theta(F;\mu_\varepsilon,\ldots,\mu_\varepsilon) = \lim_{\varepsilon \to 0} \int_{\rr^{nk}} \h{F}(-\bb{A}^t\xi) \prod_{j=1}^k \h{\mu}_\varepsilon(\xi_j)\ d\xi.
\end{equation}
By (\ref{muepsilontomu}),
\[
    \h{F}(-\bb{A}^t\xi) \prod_{j=1}^k \h{\mu}_\varepsilon(\xi_j) \to \h{F}(-\bb{A}^t\xi) \prod_{j=1}^k \h{\mu}(\xi_j)
\]
pointwise, and by (\ref{muepsilondecay}),
\[
    \left|\h{F}(-\bb{A}^t\xi) \prod_{j=1}^k \h{\mu}_\varepsilon(\xi_j)\right| \leq C|\h{F}(-\bb{A}^t\xi)| \prod_{j=1}^k (1+|\xi_j|)^{-\beta/2}.
\]
Existence of the limit in (\ref{nufourierform}) will follow from the dominated convergence theorem, if we prove $|\h{F}(\bb{A}^t\xi)| \prod_{j=1}^k (1+|\xi_j|)^{-\beta/2} \in L^1(\rr^{nk})$. To this end, let $g(t) = (1+|t|)^{\beta/2}$ for $t \in \rr^n$. Then
\begin{align*}
    \int_{\rr^{nk}} |\h{F}(\bb{A}^t\xi)| \prod_{j=1}^k (1+|\xi_j|)^{-\beta/2}\ d\xi & = \int_{\rr^m} |\h{F}(\kappa)| \int_{\bb{A}^t\xi = \kappa} \prod_{j=1}^k g(\xi_j)\ d\sigma(\xi) d\kappa \\
        & = \int_{\rr^m} |\h{F}(\kappa)|\Lambda^*_{\tau(\kappa)}(g)\ d\kappa \\
        & \leq C\int_{\rr^m} |\h{F}(\kappa)|\ d\kappa < \infty.
\end{align*}
Here $\tau(\kappa)$ is the unique vector in $S^\perp$ such that $\{\bb{A}^t\xi = \kappa\} = S+\tau(\kappa)$. We have used Proposition \ref{lambdaextbound} to bound $\Lambda^*_{\tau(\kappa)}$ in the last displayed inequality above, and used the fact that $\h{F} \in \cals(\rr^m)$ to deduce that $\h{F} \in L^1(\rr^m)$. By the dominated convergence theorem, the limit in (\ref{nudefn}) exists.


To extend $\nu$ to all of $C([0,1]^m)$, fix $F \in C([0,1]^m)$. Extend $F$ to $\widetilde{F} \in C_c(\rr^m)$ so that $F = \widetilde{F}$ on $[0,1]^m$. We will reuse $F$ to mean $\widetilde{F}$ for convenience. Get a sequence of functions $F_n \in C^\infty([0,1]^m)$ with $\h{F}_n \in \cals(\rr^m)$ such that $\|F-F_n\|_\infty \to 0$. By the preceding proof,
\[
	|\langle \nu, F_n-F_m \rangle| \leq C\|F_n-F_m\|_\infty \to 0
\]
as $n,m \to \infty$ since $F_n$ is Cauchy in sup norm. Thus the sequence of scalars $\langle \nu, F_n \rangle$ is Cauchy and hence converges. Define
\[
	\langle \nu, F \rangle = \lim_{n \to \infty} \langle \nu, F_n \rangle.
\]
Clearly,
\[
	|\langle \nu, F \rangle| = \lim_{n \to \infty}|\langle \nu, F_n \rangle| \leq C\lim_{n \to \infty}\|F_n\|_\infty = C\|F\|_\infty.
\]
\end{proof}

The proof of Lemma \ref{nulimexistsandbdd} yields the following corollary which will be used later in the sequel:

\begin{corollary}
\label{nulimexistsandbddcor}
For $F \in C([0,1]^m)$ with $\h{F} \in \cals(\rr^m)$,
\[
	\langle \nu, F \rangle = \int_{\rr^{nk}} \h{F}(-\bb{A}^t\xi) \prod_{j=1}^k \h{\mu}(\xi_j) \ d\xi.
\]
\end{corollary}


\subsection{Proof of Proposition 4.1(a)}

\begin{proof}
We have
\[
	\nu(\rr^{m}) = \langle \nu, 1 \rangle = \lim_{\varepsilon \to 0} \int_{\rr^m} \prod_{j=1}^k \mu_{\varepsilon}(A_jx)\ dx = \lim_{\varepsilon \to 0} \Lambda(\mu_\varepsilon) = \lim_{\varepsilon \to 0}\Lambda^*(\h{\mu}_\varepsilon) = \Lambda^*(\h{\mu})
\]
by Lemma \ref{lambdamuepstolambdamu}, with $\tau = 0$.
\end{proof}


\subsection{Proof of Proposition 4.1(b)}

\begin{proof}
Define
\[
	X := \{x \in \rr^m \st A_1x, \ldots, A_kx \in \supp \mu\}.
\]
Since $\supp \mu$ is closed, $X$ is closed. Let $F$ be any continuous function on $\rr^m$ with $\supp F$ disjoint from $X$, then $\dist(\supp F,X) > 0$. In order to prove that $\nu$ is supported on $X$, we aim to show that $\langle \nu, F \rangle = 0$. To this end, let us define
\begin{align*}
	X_N & := \{\vec{x} \in \rr^m \st \dist(A_j\vec{x}, \supp(\mu)) \leq 1/N \text{ for every } 1 \leq j \leq k\} \\
      & = \bigcap_{j=1}^k \{\vec{x} \in \rr^m \st \dist(A_j\vec{x}, \supp(\mu)) \leq 1/N\}.
\end{align*}
Then $X \subseteq X_N$ for every $N$, and $X = \bigcap_{N=1}^\infty X_N$. Furthermore,
\[
  X_N^c = \bigcup_{j=1}^k \{\vec{x} \in \rr^m \st \dist(A_j\vec{x}, \supp(\mu)) > 1/N\}
\]
is an open set for every $N \geq 1$, with
\[
	\supp(F) \subseteq X^c = \bigcup_{N=1}^\infty X_N^c.
\]
Introducing a smooth partition of unity subordinate to $\{X_N^c\}_N$, we can write $F = \sum_{N} F_N$ where each $F_N \in C_c^\infty(\rr^m)$ with $\supp(F_N) \subseteq X_N^c$. Note that since $\supp(F)$ is a compact subset of $X^c$, it follows from the definition of a partition of unity that the infinite sum above is in fact a finite sum, so there is no issue of convergence.

Let $\mu_\varepsilon^{A_j}(\vec{x}) := \mu_\varepsilon(A_j\vec{x})$. To compute
\begin{align*}
  \langle \nu, F_N \rangle & = \lim_{\varepsilon \to 0} \int_{\rr^m} F_N(\vec{x}) \prod_{j=1}^k \mu_\varepsilon(A_j\vec{x})\ d\vec{x} \\
      & = \lim_{\varepsilon \to 0} \int_{\rr^m} F_N(\vec{x}) \prod_{j=1}^k \mu_\varepsilon^{A_j}(\vec{x})\ d\vec{x}
\end{align*}
for a fixed $N \geq 1$, we observe that
\begin{align*}
  \supp(\mu_\varepsilon^{A_j}) & \subseteq \{\vec{x} \in \rr^m \st \dist(A_j\vec{x}, \supp(\mu)) \leq \varepsilon\} \\
     & \subseteq \{\vec{x} \in \rr^m \st \dist(A_j\vec{x}, \supp(\mu)) \leq 1/N\}
\end{align*}
if $\varepsilon \leq 1/N$. Thus the product $\prod_{j=1}^k \mu_\varepsilon^{A_j}(\vec{x})$ is supported on $X_N$, whereas $F_N$ is supported on $X_N^c$. This implies
\[
  \int_{\rr^m} F_N(\vec{x}) \prod_{j=1}^k \mu_\varepsilon^{A_j}(\vec{x})\ d\vec{x} = 0
\]
for all $\varepsilon \leq 1/N$, so that $\langle \nu, F_N \rangle = 0$ for every $N \geq 1$. Therefore, $\langle \nu, F \rangle = 0$ as claimed.
\end{proof}


\subsection{Proof of Proposition 4.1(c)}

\begin{proof}
It suffices to prove the proposition for $\dim V = v = m-1$, since smaller subspaces have even less measure. Let $P_V$ denote the projection onto $V$. Fix $v_0 \in V$ and define
\[
	V_{\delta,\gamma} = \{x \in \rr^m \st |v_0 - P_Vx| \leq \gamma, \dist(x,V) = |P_{V^\perp}x| \leq \delta\}.
\]
It suffices to prove $\nu(V_\delta) \to 0$ as $\delta \to 0$. If $\phi_\delta$ is any smooth function with
\begin{equation}
\label{phideltadefn}
	\phi_\delta = \begin{cases} 1 & \text{ on } V_{\delta,1}, \\
																0 & \text{ on } \rr^m \setminus V_{\delta,2}, \end{cases}
\end{equation}
then $\nu(V_\delta) \leq \int \phi_\delta\ d\nu = \langle \nu, \phi_\delta \rangle$, so we aim to show that $\langle \nu, \phi_\delta \rangle \to 0$ as $\delta \to 0$.

Fix bases $\{\mathfrak{a}_1,\ldots,\mathfrak{a}_{m-1}\}$ and $\{\mathfrak{b}\}$ for $V$ and $V^\perp$ respectively, such that $\{\mathfrak{a}_1,\ldots,\mathfrak{a}_{m-1},\mathfrak{b}\}$ forms an orthonormal basis of $\rr^m$. Thus, for any $x \in \rr^m$, there is a unique decomposition
\begin{align}
	\begin{split}
	x = u+w&, \text{ with } u = \sum_{j=1}^{m-1} a_j\mathfrak{a}_j \in V, w = b\mathfrak{b} \in V^\perp,\\
	& \text{ where } \vec{a} = (a_1,\ldots,a_{m-1})^t \in \rr^{m-1}, b \in \rr. \label{xdecomp}
	\end{split}
\end{align}

Without loss of generality, we may assume $\phi_\delta$ as in (\ref{phideltadefn}) to be variable-separated as
\begin{equation}
\label{separation}
	\phi_\delta(x) = \phi_V(\vec{a})\phi_{V^\perp}(\delta^{-1}b),
\end{equation}
where $\phi_V \in C_c^\infty(\rr^{m-1})$ is supported on $\{\vec{a} \st |\sum_{j=1}^{m-1} a_j\mathfrak{a}_j - v_0| \leq 2\}$ and $\phi_{V^\perp} \in C_c^\infty(\rr)$ is supported on $\{b \st |b| \leq 2\}$.

By Corollary \ref{nulimexistsandbddcor},
\begin{equation}
\label{nuphidelta}
	\langle \nu, \phi_\delta \rangle = \int_{\rr^{nk}} \h{\phi}_\delta(-\bb{A}^t\xi) \prod_{j=1}^k \h{\mu}(\xi_j)\ d\xi.
\end{equation}
We will show that this integral tends to $0$ as $\delta \to 0$. The estimation of this integral relies on an orthogonal decomposition of $\rr^{nk}$ into specific subspaces, which we now describe. Let
\begin{equation}
\label{Wdefn}
		W = \{\xi \in \rr^{nk} \st \bb{A}^t\xi \cdot x = 0 \text{ for all } x \in V\}.
\end{equation}
Then $S$ is clearly a subspace of $W$, as $\bb{A}^t\xi = 0$ if $\xi \in S$. It is also not difficult to see that $\dim W = nk-v = nk-(m-1)$. The proof of this has been relegated to Lemma \ref{dimW} below. A consequence of this fact is that
\begin{equation}
\label{wcapsperpdim}
	\dim W \cap S^\perp = 1
\end{equation}
since $S$ is $(nk-m)$-dimensional. Now write $\xi \in \rr^{nk}$ as
\[
	\xi = \zeta + \eta + \lambda, \text{ where } \zeta \in S, \eta \in W \cap S^\perp, \lambda \in W^\perp,
\]
so that $d\xi = d\sigma_{S+\eta+\lambda}(\xi) \ d\eta \ d\lambda$. Here $d\sigma_{S+\eta+\lambda}$ denotes the surface measure on $S+\eta+\lambda$, as defined in Definition \ref{surfacemeas}. We will soon show, in Lemmas \ref{hmudecay} and \ref{hphideltadecay} below, that the two factors of the integrand in (\ref{nuphidelta}) obey the size estimates:
\begin{equation}
\label{hmubound}
		\left|\prod_{j=1}^k \h{\mu}(\xi_j)\right| \lesssim (1+|\bb{A}^t(\eta+\lambda) \cdot \mathfrak{b}|)^{-\varepsilon} \prod_{j=1}^k g(\xi_j),
\end{equation}
and
\begin{equation}
\label{decayrate}
		|\h{\phi}_\delta(-\bb{A}^t\xi)| \leq \delta C_M(1+|\lambda|)^{-M}(1+\delta|\bb{A}^t(\eta+\lambda) \cdot \mathfrak{b}|)^{-M}
\end{equation}
for any $M \geq 1$. Here $g(\xi_j) = (1+|\xi_j|)^{-\beta/2+\varepsilon}$, where $\varepsilon > 0$ is chosen sufficiently small so that
\begin{equation}
\label{betaprime}
	\beta-2\varepsilon > 2(nk-m)/k.
\end{equation}
Notice this is possible since $\beta > 2(nk-m)/k$.

Assuming (\ref{hmubound}) and (\ref{decayrate}) temporarily, the estimation of (\ref{nuphidelta}) proceeds as follows.
\begin{align*}
	& |\langle \nu, \phi_\delta \rangle| \leq \int_{\rr^{nk}} \left|\h{\phi}_\delta(-\bb{A}^t\xi) \prod_{j=1}^k \h{\mu}(\xi_j)\right|\ d\xi \\
			\lesssim \ & \int_{W^\perp} \left[\int_{W \cap S^\perp} \left[\int_{S+\eta+\lambda} \prod_{j=1}^k g(\xi_j)\ d\sigma_{S+\eta+\lambda}(\xi)\right] J(\eta,\lambda) \ d\eta\right] (1+|\lambda|)^{-M} \ d\lambda,
\end{align*}
where
\[
	J(\eta,\lambda) = \delta C_M (1+|\bb{A}^t(\eta+\lambda) \cdot \mathfrak{b}|)^{-\varepsilon}(1+\delta|\bb{A}^t(\eta+\lambda) \cdot \mathfrak{b}|)^{-M}.
\]
We claim that
\begin{equation}
\label{innerintg}
	\int_{S+\eta+\lambda} \prod_{j=1}^k g(\xi_j)\ d\sigma(\xi) \leq C
\end{equation}
and
\begin{equation}
\label{midintJ}
	\sup_{\lambda \in W^\perp} \int_{W \cap S^\perp} J(\eta,\lambda) \ d\eta \leq C_M\delta^{\varepsilon/2}.
\end{equation}
These two estimates yield, for $M \geq \dim(W^\perp)+1$,
\begin{align*}
	|\langle \nu, \phi_\delta \rangle| & \lesssim C_M\delta^{\varepsilon/2} \int_{W^\perp} (1+|\lambda|)^{-M} \ d\lambda \\
			& \lesssim C_M\delta^{\varepsilon/2} \to 0
\end{align*}
as $\delta \to 0$, as required.

It remains to establish the estimates in (\ref{innerintg}) and (\ref{midintJ}). For the former, we observe that the left hand side of the inequality is $\Lambda^*_{\eta+\lambda}(g)$, so the desired conclusion follows from Proposition \ref{lambdaextbound} and our choice (\ref{betaprime}) of $\varepsilon$.

To prove (\ref{midintJ}), we recall (\ref{wcapsperpdim}) so we may parametrize $\eta = sw_0$ for some fixed unit vector $w_0 \in W \cap S^\perp \setminus \{0\}$, with $d\eta = ds$. To confirm that $J(\eta,\lambda)$ has decay in $\eta$, we need to verify that $\bb{A}^tw_0 \cdot \mathfrak{b} \neq 0$. Indeed, if $\bb{A}^tw_0 \cdot \mathfrak{b} = 0$, then $\bb{A}^tw_0 \in V$ since $\mathfrak{b} \in V^\perp$. Since $w_0$ also lies in $W$ given by (\ref{Wdefn}), this implies $\bb{A}^tw_0 \cdot \bb{A}^tw_0 = 0$, so that $\bb{A}^tw_0 = 0$. But the last equation says $w_0 \in S$, whereas $w_0 \in S^\perp$ by assumption. This forces $w_0 = 0$, a contradiction to $\|w_0\| = 1$.

We now set $c_0 := \bb{A}^tw_0 \cdot \mathfrak{b}$ which is nonzero by the discussion in the preceding paragraph. Making a linear change of variable $t = sc_0 + \bb{A}^t\lambda \cdot \mathfrak{b}$, with Jacobian $ds = dt/c_0$, we proceed to estimate the integral in (\ref{midintJ}) by partitioning the region of integration as follows,
\begin{align*}
	& \int_{W \cap S^\perp} J(\eta,\lambda) \ d\eta \\
	= \ & \delta C_M \int_\rr (1+|t|)^{-\varepsilon}(1+\delta|t|)^{-M} \frac{dt}{c_0} \nonumber \\
	= \ & \delta C_M \int_{|t| \leq \delta^{-1/2}} (1+|t|)^{-\varepsilon}(1+\delta|t|)^{-M}\ dt \nonumber \\
	& \qquad + \delta C_M\int_{|t| > \delta^{-1/2}} (1+|t|)^{-\varepsilon}(1+\delta|t|)^{-M}\ dt \nonumber \\
	\lesssim \ & \delta C_M\int_{|t| \leq \delta^{-1/2}} 1\ dt + \delta^{\varepsilon/2}\delta C_M\int_{|t| > \delta^{-1/2}} (1+\delta|t|)^{-M}\ dt \nonumber \\
	\lesssim \ & \delta^{1/2} C_M + \delta^{\varepsilon/2}\delta C_M\int_{\rr} (1+\delta|\eta|)^{-M}\ d\eta \nonumber \\
	\approx \ & \delta^{1/2}C_M + \delta^{\varepsilon/2}C_M \lesssim \delta^{\varepsilon/2}C_M.
\end{align*}
This completes the proof of (\ref{midintJ}) and hence the proof of the proposition.
\end{proof}

Now we prove the three lemmas required earlier for this proof.

\begin{lemma}
\label{dimW}
Define $W$ as in (\ref{Wdefn}). Then $\dim W = nk-v = nk-(m-1)$.
\end{lemma}

\begin{proof}
As before, let $P_V$ denote the projection onto $V$. By (\ref{Wdefn}),
\begin{align*}
	W & = \{\xi \in \rr^{nk} \st \bb{A}^t\xi \cdot x = 0 \text{ for all } x \in V\} \\
		& = \{\xi \in \rr^{nk} \st \bb{A}^t\xi \cdot P_Vx = 0 \text{ for all } x \in \rr^m\} \\
		& = \{\xi \in \rr^{nk} \st P_V^t\bb{A}^t\xi \cdot x = 0 \text{ for all } x \in \rr^m\} \\
		& = \{\xi \in \rr^{nk} \st P_V^t\bb{A}^t\xi\} \\
		& = \mathcal{N}(P_V^t\bb{A}^t),
\end{align*}
Writing $\rr^m$ as $V \oplus V^\perp$, the dimension of $\bb{A}(V)$ must be equal to $\dim V = m-1$ as $\bb{A}$ is of full rank and hence an isomorphism from $\rr^m$ to the range of $\bb{A}$. Then $\bb{A}P_V$ is an isomorphism from $V$ to the range of $\bb{A}P_V$, so $\mathrm{rank}(P_V^t\bb{A}^t) = \dim(V) = m-1$, and thus $\dim W = nk-(m-1)$.
\end{proof}

\begin{lemma}
\label{hmudecay}
With $\mu, \xi, \eta, \lambda, \beta$ defined as in the proof of Proposition \ref{multlinmeas}(c), we have
\[
	\left|\prod_{j=1}^k \h{\mu}(\xi_j)\right| \lesssim (1+|\bb{A}^t(\eta+\lambda) \cdot \mathfrak{b}|)^{-\varepsilon} \prod_{j=1}^k (1+|\xi_j|)^{-\beta/2+\varepsilon},
\]
for any $\varepsilon > 0$.
\end{lemma}

\begin{proof}
The decay condition (\ref{mudecay}) on $\h{\mu}$ gives
\begin{align*}
	\left|\prod_{j=1}^k \h{\mu}(\xi_j)\right| & \leq C\prod_{j=1}^k(1+|\xi_j|)^{-\varepsilon} \prod_{j=1}^k(1+|\xi_j|)^{-\beta/2+\varepsilon},
\end{align*}
for any $\varepsilon > 0$. We have
\[
	|\eta+\lambda| \leq |\xi| \leq k \max_{1 \leq j \leq k} |\xi_j|,
\]
and by Cauchy-Schwarz
\[
	|\bb{A}^t(\eta+\lambda) \cdot \mathfrak{b}| = |(\eta+\lambda) \cdot \bb{A}\mathfrak{b}| \leq |\eta+\lambda||\bb{A}\mathfrak{b}|.
\]
Since $\bb{A}\mathfrak{b}$ is fixed, $|\bb{A}^t(\eta+\lambda) \cdot \mathfrak{b}| \lesssim |\eta+\lambda|$, and so
\begin{align*}
	\left|\prod_{j=1}^k \h{\mu}(\xi_j)\right| & \lesssim (1+|\eta+\lambda|)^{-\varepsilon} \prod_{j=1}^k(1+|\xi_j|)^{-\beta/2+\varepsilon} \\
			& \lesssim (1+|\bb{A}^t(\eta+\lambda) \cdot \mathfrak{b}|)^{-\varepsilon} \prod_{j=1}^k(1+|\xi_j|)^{-\beta/2+\varepsilon}.
\end{align*}
\end{proof}

\begin{lemma}
\label{hphideltadecay}
With $\phi_\delta, \xi, \zeta, \eta, \lambda, \beta$ defined as in the proof of Proposition \ref{multlinmeas}(c), we have
\[
	|\h{\phi}_\delta(-\bb{A}^t\xi)| \leq \delta C_M(1+|\lambda|)^{-M}(1+\delta|\bb{A}^t(\eta+\lambda) \cdot \mathfrak{b}|)^{-M},
\]
for any $M \in \rr$.
\end{lemma}

\begin{proof}
	Since $\zeta \in S$ implies $\bb{A}^t\zeta = 0$, we have
	\begin{align*}
		\h{\phi}_\delta(-\bb{A}^t\xi) & = \h{\phi}_\delta(-\bb{A}^t(\zeta + \eta + \lambda)) \\
				& = \h{\phi}_\delta(-\bb{A}^t(\eta + \lambda)) \\
				& = \iint \phi_V(\vec{a})\phi_{V^\perp}(\delta^{-1}b)e^{2\pi i\bb{A}^t(\eta+\lambda) \cdot (u+w)}\ du \ dw.
	\end{align*}
	By definition, $\eta \in W$ and $u \in V$ give $\bb{A}^t\eta \cdot u = 0$, and so
	\begin{align*}
		\h{\phi}_\delta(-\bb{A}^t\xi) & = \left[\int_V \phi_V(\vec{a})e^{2\pi i\bb{A}^t\lambda \cdot A_0\vec{a}}\ du\right]\left[\int_{V^\perp}\phi_{V^\perp}(\delta^{-1})e^{2\pi ib\bb{A}^t(\eta+\lambda) \cdot b\mathfrak{b}}\ dw\right] \\
				& = \left[\int_{\rr^{m-1}} \phi_V(\vec{a})e^{2\pi iA_0^t\bb{A}^t\lambda \cdot \vec{a}}\ d\vec{a}\right]\left[\int_{\rr}\phi_{V^\perp}(\delta^{-1}b)e^{2\pi ib\bb{A}^t(\eta+\lambda) \cdot \mathfrak{b}}\ db\right],
	\end{align*}
	where we have used $u = \sum_{i=1}^{m-1} a_i \mathfrak{a}_i = A_0\vec{a}$ for some matrix $A_0$, $w = b\mathfrak{b}$, and $du \ dw = da \ db$.

The first factor is by definition $\h{\phi}_V(-A_0^t\bb{A}^t\lambda)$. Since $\h{\phi}_V \in \cals(\rr^{m-1})$, for every $M \in \rr$ we have
\[
	|\h{\phi}_V(-A_0^t\bb{A}^t\lambda)| \leq C_M(1+|A_0^t\bb{A}^t\lambda|)^{-M}.
\]
We claim $|A_0^t\bb{A}^t\lambda| \gtrsim |\lambda|$ for all $\lambda \in W^\perp$. Since $A_0^t\bb{A}^t$ is linear, it suffices to prove $A_0^t\bb{A}^t\lambda \neq 0$ for any $\lambda \in W^\perp$. If $\lambda \in W^\perp$, then by definition of $W^\perp$ there exists $x \in V \setminus \{0\}$ such that $(\bb{A}^t\lambda,x) \neq 0$. Then $x = A_0\vec{a}$ for some $\vec{a} \neq 0$, so
\[
	(A_0^t\bb{A}^t\lambda,\vec{a}) = (\bb{A}^t\lambda,A_0^t\vec{a}) \neq 0,
\]
and hence $A_0^t\bb{A}^t\lambda \neq 0$. Then for every $M \in \rr$,
\[
	|\h{\phi}_V(-A_0^t\bb{A}^t\lambda)| \leq C_M(1+|\lambda|)^{-M}.
\]

The second factor is, upon scaling, $\delta\h{\phi}_{V^\perp}(-\delta \bb{A}^t(\eta+\lambda) \cdot \mathfrak{b})$. As $\h{\phi}_{V^\perp} \in \cals(\rr)$, for every $M \in \rr$ we have
\[
	|\delta\h{\phi}_{V^\perp}(-\delta \bb{A}^t(\eta+\lambda) \cdot \mathfrak{b})| \leq \delta (1+\delta|\bb{A}^t(\eta+\lambda) \cdot \mathfrak{b}|)^{-M},
\]
completing the proof.
\end{proof}



\section{Absolutely continuous estimates}\label{ch:AbsContEstimates}

In this section, we we will restrict to the case when $A_j$ is of the form $A_j = (I_{n \times n} \ B_j)$, where $B_j$ are $n \times (m-n)$ matrices. Set $x\in\rr^n$ and $y\in \rr^{m-n}$, so that our configurations are of the form $\{x+B_1y,\ldots,x+B_ky\}$. With $r$ defined as in (\ref{rdefn}), we will also assume $k-1 \geq 2r$, or equivalently, $n \lceil (k+1)/2 \rceil \leq m$, the first inequality of condition (\ref{nmkcond}) of the main result.

\begin{proposition}
\label{absctslowerbd}
For every $\delta,M > 0$, there exists a constant $c(\delta,M) > 0$ with the following property: for every function $f: [0,1]^n \to \rr$, $0 \leq f \leq M$, $\int f \geq \delta$, we have $\Lambda(f) \geq c(\delta,M)$.
\end{proposition}

We will proceed as in the proof of Varnavides' Theorem given in \cite{tao}. The strategy will be to decompose $f = g+b$ into a ``good'' function $g$ which is the major contribution and a ``bad'' function $b$ whose contribution is negligible. This will be made precise in the following subsection.


\subsection{Preliminaries}

\begin{proposition}
\label{decompapprox}
Let $f$ be as in Proposition \ref{absctslowerbd}. Suppose $f = g+b$ where
\[
	\|g\|_\infty, \|b\|_\infty \leq M; \quad \|g\|_1,\|b\|_1 = \delta.
\]
Then
\[
	\Lambda(f) = \Lambda(g) + O(C(M,\delta)\|\h{b}\|_\infty).
\]
\end{proposition}

\begin{proof}
We use the decomposition $f = g+b$ and the linearity of $\Lambda$ to decompose $\Lambda(f)$ into $2^k$ pieces. The main piece will be $\Lambda(g)$ and the remaining pieces which constitute the error term have at least one copy of $b$. By the hypothesis and H\"{o}lder's inequality,
\[
	\|g\|_2^2,\|b\|_2^2 \leq M\delta.
\]
We will apply Lemma \ref{holdertypebound} below to estimate each of the $2^k-1$ summands in the error term, arriving at an upper bound of $(2^{k}-1)\|\h{b}\|_\infty(M\delta)^{r}$.
\end{proof}

We now prove the lemma required for the previous proposition.

\begin{lemma}
\label{holdertypebound}
Let $f_j$ be as in Proposition \ref{absctslowerbd}.
Assume moreover that $k-1 \geq 2r$ and $\|f_j\|_1 \leq 1$ for $1 \leq j \leq k$. Then
\[
	|\Lambda(f_1,\ldots,f_k)| \leq M \|\h{f}_k\|_\infty\|f_r\|_1^{1/2}\|f_{2r}\|_1^{1/2}\prod_{\substack{j=1 \\ j \neq r}}^{2r-1}\|\h{f}_j\|_2.
\]
We have a similar bound for permutations of $f_1,\ldots,f_k$.
\end{lemma}

\begin{proof}
Let us recall the Fourier representation of $\Lambda$ from Proposition \ref{fourierform}, which gives
\[
	|\Lambda(f_1,\ldots,f_k)| \leq \int_S \prod_{j=1}^k |\h{f}_j(\xi_j)| \ d\sigma(\xi).
\]
Since $\|\h{f}_j\|_\infty \leq \|f_j\|_1 \leq 1$ for each $j$, reducing the number of factors in the product that appears in the last integrand only makes the integral larger. We use the hypothesis $k-1 \geq 2r$ to drop $(k-1-2r)$ of these factors and split the remaining $2r$ into two groups and apply the Cauchy-Schwarz inequality. Executing these steps leads to
\begin{align*}
	|\Lambda(f_1,\ldots,f_k)| & \leq \int_S \prod_{j=1}^k |\h{f}_j(\xi_j)| \ d\sigma(\xi) \\
			& \leq \|\h{f}_k\|_\infty \int_S \prod_{j=1}^r |\h{f}_j(\xi_j)| \prod_{j=r+1}^{2r} |\h{f}_j(\xi_j)| \ d\sigma(\xi) \\
			& \leq \|\h{f}_k\|_\infty \left(\int_S \prod_{j=1}^r |\h{f}_j(\xi_j)|^2 \ d\sigma(\xi)\right)^{1/2} \left(\prod_{j=r+1}^{2r} |\h{f}_j(\xi_j)|^2 \ d\sigma(\xi)\right)^{1/2}.
\end{align*}
Both of the above integrals are estimated in the same way; we will focus only on the first. If $\xi_j = (\xi_{j,1},\ldots,\xi_{j,n})$, let $\xi_j' = (\xi_{j,1},\ldots,\xi_{j,n'})$ with $n'$ as defined in Definition \ref{nprimedefn}; notice this is the same as $\xi_j^{\mathrm{id}}$ as defined in Proposition \ref{lambdaextbound}. By Lemma \ref{restrictednorm} below, $\|\h{f}_r\|^2_{L^2_{\xi'_r}} \leq M\|f_r\|_1$, and so by Lemma \ref{lembasis},
\begin{align*}
	& \int_S \prod_{j=1}^r |\h{f}_j(\xi_j)|^2 \ d\sigma(\xi) \\
	= \ & \int_S \prod_{j=1}^r |\h{f}_j(\xi_j)|^2 \ d\xi_r' d\xi_1 \cdots d\xi_{r-1}\\
	\leq \ & \int_{\rr^{n(r-1)}} M \|f_r\|_1 \prod_{j=1}^{r-1} |\h{f}_j(\xi_j)|^2 \ d\xi_1 \cdots d\xi_{r-1} \\
	= \ & M \|f_r\|_1\prod_{j=1}^{r-1} \|\h{f}_j\|_2^2.
\end{align*}
The result follows.
\end{proof}

\begin{lemma}
\label{restrictednorm}
Let $1 \leq j \leq k$. If $\xi_j = (\xi_{j,1},\ldots,\xi_{j,n})$, we denote $\xi_j' = (\xi_{j,1},\ldots,\xi_{j,n'})$ and $\xi_j'' = (\xi_{j,n'+1},\ldots,\xi_{j,k})$. Suppose
\begin{enumerate}[(a)]
	\item $|f_j| \leq M$,
	\item $\supp f_j \subseteq [0,1]^n$.
\end{enumerate}
Then
\[
	\|\h{f}_j\|^2_{L^2_{\xi'}} \leq M \|f_j\|_{1}
\]
uniformly for all $\xi_j''$.
\end{lemma}

\begin{proof}
Fix $\xi_j''$, and let $F(x',\xi_j'') = \int f_j(x',x'')e^{-2\pi i \xi_j'' \cdot x''} \ dx''$, where $x',x''$ are the dual variables to $\xi_j',\xi_j''$ respectively. Then (a) and (b) give $|F(x',\xi_j'')| \leq M$ for all $x'$, and we calculate
\[
	\int |F(x',\xi_j'')| \ dx' \leq \iint |f_j(x',x'')| \ dx' dx'' = \|f_j\|_1.
\]
By H\"{o}lder's inequality, $\|F\|^2_{L^2_{x'}} \leq \|F\|_\infty\,\|F\|_{L^1_{x'}} \leq M \|f_j\|_1$. Now,
\[
	\h{f}_j(\xi_j',\xi_j'') = \int F(x',\xi_j'')e^{-2\pi i \xi_j' \cdot x'} \ dx',
\]
which is the Fourier transform of $F$ in $x'$. Therefore by Plancherel's theorem in the $x'$ variables,  $\|\h{f}_j\|^2_{L^2_{\xi_j'}} = \|F\|^2_{L^2_{x'}} \leq M \|f_j\|_{1}$.
\end{proof}


\subsection{Almost periodic functions}

In light of Proposition \ref{decompapprox}, our next goal will be to identify a large class of ``good" functions $g$ for which we can bound $\Lambda(g)$ from below. It turns out that almost periodic functions, defined analogously to \cite{tao}, can be used for this purpose.

\begin{definition}
\begin{enumerate}
	\item A \emph{character} is a function $\chi: [0,1]^n \to \cc$ of the form $\chi(x) = e^{2\pi iv \cdot x}$ for some $v \in \zz^n$.
	\item If $K \in \nn$, then a $K$-\emph{quasiperiodic function} is a function $f$ of the form $\sum_{\ell=1}^K c_\ell\chi_\ell$ where each $\chi_\ell$ are characters (not necessarily distinct), and $c_\ell$ are scalars with $|c_\ell| \leq 1$.
	\item If $\sigma > 0$, then $f: [0,1]^n \to \cc$ is $(\sigma,K)$-\emph{almost periodic} if there exists a $K$-quasiperiodic function $f_{QP}$ such that $\|f-f_{QP}\|_{L^2([0,1]^n)} \leq \sigma$. We call $f_{QP}$ a \emph{K-quasiperiodic function approximating} $f$ \emph{within} $\sigma$.
\end{enumerate}
\end{definition}

\begin{lemma}
\label{APrec}
Let $K \in \nn$, $M > 0$, $0 < \delta < 1$, and
\begin{equation}
\label{e-sigma}
0 < \sigma \leq \frac{\delta^k}{4kM^{k-1}}.
\end{equation}
Then there exists $c(K,\delta,M) > 0$ such that for any non-negative $(\sigma, K)$-almost periodic function $f$ bounded by $M$ and obeying $\int f \geq \delta$,
\[
	\Lambda(f) \geq c(K,\delta,M).
\]
\end{lemma}

\begin{proof}
Our goal is to bound $\Lambda(f)$ from below by a multiple of
$\|f\|_1^k$, which is known to be at least as large as $\delta^k$. We
will achieve this by approximating each factor in the integral
defining $\Lambda$ by $f$, on a reasonably large set with acceptable
error terms.

To this end, let $f_{QP}$ be a $K$-quasiperiodic function approximating $f$ within $\sigma$, say $f_{QP}(x) = \sum_{\ell=1}^K c_\ell e^{2\pi iv_\ell \cdot x}$ and $\|f-f_{QP}\|_2 \leq \sigma$. Let $\varepsilon > 0$ be a small constant to be fixed later, and define
\begin{equation}
\label{cepsdefn}
	C_\varepsilon = \{y \in \rr^{m-n} \st \|A_j^tv_\ell \cdot y\| \leq \varepsilon, \text{ for all } 1 \leq j \leq k, 1 \leq \ell \leq K\},
\end{equation}
where $\|t\|$ denotes the distance of $t \in \rr$ to the nearest
integer. We shall prove in Corollary \ref{cepsmeas} that
\begin{equation}
\label{cepsbd}
	|C_\varepsilon| \geq c(\varepsilon,K) > 0
\end{equation}
for some $c(\varepsilon,K)$ possibly depending on $k$, $m$ and $n$ but
is independent of $f$. Let $T^a$ be the shift map $T^af(x) :=
f(x+a)$. For $y \in C_\varepsilon$ and any $x \in [0,1]^n$,
\begin{align}
	|T^{B_jy}f_{QP}(x)-f_{QP}(x)| & = \left|\sum_{\ell=1}^K c_\ell e^{2\pi iv_\ell \cdot x}(1-e^{2\pi iv_\ell \cdot A_jy})\right| \nonumber \\
			& \leq \sum_{\ell=1}^K |1-e^{2\pi iA_j^tv_\ell \cdot y}| \nonumber \\
			& \leq \sum_{\ell=1}^K |A_j^tv_\ell \cdot y| \nonumber \\
			& \leq K\varepsilon. \label{TB_jbound}
\end{align}
The bound above leads to the following estimate:
\begin{align*}
	\|T^{B_jy}f-f\|_{L^1_x} & \leq \|T^{B_jy}f-T^{B_jy}f_{QP}\|_{L^2_x} +
        \|T^{B_jy}f_{QP}-f_{QP}\|_{L^\infty_x} + \|f_{QP}-f\|_{L^2_x} \\
      & = 2\|f_{QP} - f\|_{L_x^2} + \|T^{B_jy}f_{QP} - f_{QP}\|_{L_x^{\infty}} \\
			& \leq 2\sigma + K\varepsilon \\
			& \leq \frac{\delta^k}{2kM^{k-1}} + K\varepsilon.
\end{align*}
In the sequence of inequalities above, we have used H\"older's
inequality in the first step, triangle inequality in the second step,
norm-invariance of the shift operator in the third step, and
(\ref{TB_jbound}) in the last step.

We now choose $\varepsilon = \delta^k/(4kM^{k-1})$, so that
\begin{equation}
\label{shiftapprox}
	\|T^{B_jy}f-f\|_{L^1_x} \leq \frac{3\delta^k}{4kM^{k-1}}.
\end{equation}
For every $1 \leq j \leq k$, the bound $\|T^{B_jy}f\|_{L_x^\infty} =
\|f\|_\infty \leq M$ holds trivially,
so by Lemma \ref{prodapprox} with $C = M$, $f_j = f$, $g_j = T^{B_jy}f$, $R = k$, $p = 1$, and $\kappa = (3\delta^k)/(4kM^{k-1})$, we have
\[
	\left\|\prod_{j=1}^k T^{B_jy}f - f^k\right\|_{L_x^1} \leq kM^{k-1}\frac{3\delta^k}{4kM^{k-1}} = \frac{3\delta^k}{4}
\]
using (\ref{shiftapprox}). On the other hand, the bounded
non-negativity of $f$, the hypothesis $\int f \geq \delta$, and H\"{o}lder's
inequality lead to
\[
	\|f^k\|_1 \geq \|f\|_1^k \geq \delta^k,
\]
and so for $y \in C_\varepsilon$,
\begin{equation}
\label{shiftlowerbd}
	\left\|\prod_{j=1}^k T^{B_jy}f\right\|_{L_x^1} \geq \left\|f^k\right\|_1 - \left\|\prod_{j=1}^k T^{B_jy}f - f^k\right\|_{L_x^1} \geq \frac{\delta^k}{4}.
\end{equation}
We now combine \ref{cepsbd}, the positivity of $f$ and the above to
obtain
\begin{align*}
	\Lambda(f) & = \int_{\rr^{m-n}}\int_{\rr^n} \prod_{j=1}^k f_j(x+B_jy)\ dx \ dy \\
			& = \int_{\rr^{m-n}}\left\|\prod_{j=1}^k T^{B_jy}f\right\|_{L_x^1} \ dy \\
			& \geq \int_{C_\varepsilon} \left\|\prod_{j=1}^k T^{B_jy}f\right\|_{L_x^1} \ dy \\
			& \geq \delta^kc(\varepsilon,K)/4 = c(K,\delta,M).
\end{align*}
\end{proof}

\subsection{Ubiquity of almost periodic functions}

To make use of Lemma \ref{APrec}, we will approximate a general
function $f$ by an almost periodic function. In the following sequence
of lemmas, we construct an increasingly larger family of
$\sigma$-algebras with the property that any function measurable with
respect to these will be almost periodic. We do this by an iterative
random mechanism, the building block of which is summarized in the
next result.

\begin{lemma}
\label{sigalg}
Let $0 < \varepsilon \ll 1$ and let $\chi$ be a character. Viewing $\cc$ as $\rr^2$, partition the complex plane $\cc = \bigcup_{Q \in \bb{Q}_\varepsilon} Q$ into squares of side-length $\varepsilon$ with corners lying in the lattice $\varepsilon\zz^2$. For $\omega \in [0,1]^2$, define $\calb_{\varepsilon,\chi,\omega}$ to be the $\sigma$-algebra generated by the atoms
\[
	\{\chi^{-1}(Q + \varepsilon \omega) \st Q \in \bb{Q}_\varepsilon\}.
\]
There exists $\omega$ such that
\begin{enumerate}
\item $\|\chi - \bb{E}(\chi | \calb_{\varepsilon,\chi,\omega})\|_\infty \leq C \varepsilon$.
\item For every $\sigma > 0$ and $M < 0$, there exists $K = K(\sigma, \varepsilon,M)$ such that every function $f$ which is measurable with respect to $\calb_{\varepsilon,\chi,\omega}$ with $\|f\|_\infty \leq M$ is $(\sigma,K)$-almost periodic.
\end{enumerate}
\end{lemma}

\begin{proof}
(1) follows from definition of $\calb_{\varepsilon,\chi,\omega}$, for any $\omega$.

To prove (2), it suffices to prove: for each integer $\ell > \ell_0$ for some sufficiently large $\ell_0$, there exists a set $\Omega_\ell \subseteq [0,1]^2$, $|\Omega_\ell| > 1-C2^{-\ell}\varepsilon$ with the following property. For $\omega \in \Omega_\ell$, there exists $K = K(\sigma, \varepsilon,M)$ such that every function $f$ which is measurable with respect to $\calb_{\varepsilon,\chi,\omega}$ with $\|f\|_\infty \leq M$ is $(\sigma,K)$-almost periodic, for $\sigma = 2^{-\ell}$.

Indeed, taking $\Omega = \bigcap_{\ell > \ell_0} \Omega_\ell$, we have
\[
	|\Omega| > 1-C\sum_{\ell > \ell_0} 2^{-\ell}\varepsilon \geq 1-C2^{-\ell_0}\varepsilon
\]
and so we may find $\omega \in \Omega$. Then if $\sigma > 0$, get $\ell > \ell_0$ with $2^{-\ell} \leq \sigma$, and by the above, there exists $K = K(2^{-\ell}, \varepsilon,M)$ such that every function $f$ which is measurable with respect to $\calb_{\varepsilon,\chi,\omega}$ with $\|f\|_\infty \leq M$ is $(2^{-\ell},K)$-almost periodic.

Fix $\sigma = 2^{-\ell}$. We prove in Lemma \ref{finitelymanyatoms}
that $\calb_{\varepsilon,\chi,\omega}$ has at most $C\varepsilon^{-1}$
atoms, which we use to reduce the proof of Lemma \ref{sigalg}
to a simpler form. Namely, it suffices to prove that for $f$ an indicator function of
one of those atoms, say
\[
	f(x) = f_{Q,\omega}(x) := 1_{\chi^{-1}(Q+\varepsilon\omega)}(x) = 1_Q(\chi(x)-\varepsilon\omega),
\]
there exists a $C(\sigma,\varepsilon)$-quasiperiodic function
$g_{Q,\omega}$ such that $\|f_{Q,\omega}-g_{Q,\omega}\|_2 \leq
C^{-1}\sigma\varepsilon$ with probability
$1-C\sigma\varepsilon^{-1}$. (Here and below, $C(\sigma, \varepsilon)$
will denote a constant which may change from line to line, but always
depends only on $M,\sigma,\varepsilon$, in particular remains
independent of $f$.) Indeed, if this were the case, then any measurable $f$ may be written as
\[
	f(x) = \sum_{i \in I} c_if_{Q_i,\omega}(x)
\]
where $f_{Q_i,\omega}(c) = 1_{Q_i}(\chi(x)-\varepsilon\omega)(x)$, $Q_i \in \bb{Q}_\varepsilon$ are distinct, and $\# I \leq C\varepsilon^{-1}$. Notice $\|f\|_\infty \leq M$ and the fact that $f_{Q_i,\omega}$ have disjoint support means $|c_i| \leq M$ for $i \in I$. Letting $g_{Q_i,\omega}$ be a $C(\sigma/M,\varepsilon)$-quasiperiodic function approximating $f_{Q_i,\omega}$ to within $C^{-1}M^{-1}\sigma\varepsilon$, we have that $g = \sum_{i \in I} c_ig_{Q_i,\omega}$ is $C(\sigma,\varepsilon)$-quasiperiodic (repeating $g_{Q_i,\omega}$ at most $M$ times if necessary) and
\[
	\|f-g\|_2 \leq \sum_{i \in I} |c_i|\|f_{Q_i,\omega}-g_{Q_i,\omega}\|_2 \leq \sigma.
\]

Thus we restrict to the case when $f$ is the indicator function of one of those atoms, which we denote $f_\omega$ for ease.
\[
	f(x) = f_\omega(x) := 1_Q(\chi(x)-\varepsilon\omega).
\]
Let $0 \leq h \leq 1$ be a continuous function on $[-2,2]^2$ which is a
good $L^2$-approximation for $1_Q$; precisely,
\begin{equation} \label{def-homega}
\|1_Q - h\|_{L^2([-2,2]^2)}^2 <  \frac{\sigma^3
  \varepsilon^3}{10CM^2}.
\end{equation}
By the Weierstrass Approximation Theorem, there exists a polynomial $P$ such that
\begin{equation}
\label{WATpoly}
	\|h-P\|_{L^\infty([-2,2]^2)} < \left(\frac{\sigma^3\varepsilon^3}{10CM^2}\right)^{1/2}.
\end{equation}
Let $h_{\omega}(x) := h(\chi(x) - \varepsilon \omega)$, and
$g_\omega(x) = P(\chi(x)-\varepsilon\omega)$; notice $g_\omega$ can be
written as a linear combination of at most $C(\sigma,\varepsilon)$
characters, with coefficients at most
$C(\sigma,\varepsilon)$. Repeating characters if necessary, we may
reduce the coefficients to be less than $1$ and so $g_\omega$ is
$C(\sigma,\varepsilon)$-quasiperiodic. It should also be noted that
$\|g_\omega\|_\infty \leq ||g_{\omega} - h_{\omega}||_{\infty} +
||h_{\omega}||_{\infty} \leq 2$.

It remains to show $\|f_\omega-g_\omega\|_2 \leq C^{-1}M^{-1}\sigma\varepsilon$ with probability at least $1-C\sigma\varepsilon^{-1}$. Define
\[
	F(\omega) = \|f_\omega-g_\omega\|_2^2 = \int_{[0,1]^n} |f_\omega(x)-g_\omega(x)|^2 \ dx.
\]
By an application of Cauchy-Schwarz inequality on the integrand,
combined with (\ref{def-homega}) and Tonelli's Theorem, we obtain
\begin{align*}
	\|F\|_{L^1_{\omega}} & = \int_{[0,1]^2}\int_{[0,1]^n}
        |f_\omega(x)-g_\omega(x)|^2 \ dx \ d\omega \\
&\leq 2 \int_{[0,1]^2} \int_{[0,1]^n} \bigl[|f_{\omega}(x) - h_{\omega}(x)|^2 +
|h_{\omega}(x) - g_{\omega}(x)|^2 \bigr] \, dx \, d\omega \\
			& < \frac{\sigma^3
  \varepsilon^3}{2CM^2} + \int_{[0,1]^n}\int_{[0,1]^2}
|h_\omega(x)-g_\omega(x)|^2 \ d\omega \ dx \\ &< \frac{\sigma^3
  \varepsilon^3}{2CM^2} + \int_{[0,1]^n}\int_{[0,1]^2}
|h(\chi(x) - \varepsilon \omega)-P(\chi(x) - \varepsilon \omega)|^2 \ d\omega \ dx.
\end{align*}
Using the change of variables $\omega' = \chi(x)-\varepsilon\omega$ for fixed $x \in [0,1]^n$, we have $d\omega' = \varepsilon^2 \ d\omega$. Notice $\omega'$ belongs to $[0,\varepsilon]^2$ shifted by $\chi(x)$, so is contained in $[-2,2]^2$. By (\ref{WATpoly}),
\begin{align*}
	\|F\|_1 & \leq \frac{\sigma^3
  \varepsilon^3}{2CM^2}  + \int_{[0,1]^n}\int_{\omega' \in [-2,2]^2} |h(\omega')-P(\omega')|^2 \ \frac{d\omega'}{\varepsilon^2} \ dx \\
			& \leq \frac{\sigma^3
  \varepsilon^3}{2CM^2}  + \int_{[0,1]^n}\int_{\omega' \in [-2,2]^2} |h(\omega')-P(\omega')|^2 \ \frac{d\omega'}{\varepsilon^2} \ dx \\
			& = \frac{\sigma^3
  \varepsilon^3}{2CM^2}  + \int_{[0,1]^n} \frac{1}{\varepsilon^2}\|h-P\|_{L^2([-2,2]^2)}^2 \ dx \\
			& \leq \frac{\sigma^3
  \varepsilon^3}{2CM^2}  + \int_{[0,1]^n} \frac{1}{\varepsilon^2}\left(\frac{\sigma^3\varepsilon^3}{10CM^2}\right) \ dx < \frac{1}{\varepsilon^2}\left(\frac{\sigma^3\varepsilon^3}{CM^2}\right).
\end{align*}
By Markov's inequality,
\begin{align*}
|\{\omega \in [0,1]^2 \st \|f_\omega-g_\omega\|_2^2 >
C^{-2}M^{-2}\sigma^2\varepsilon^2\}| &=	|\{\omega \in [0,1]^2 \st
F(\omega) > C^{-2}M^{-2}\sigma^2\varepsilon^2\}| \\ & \leq \frac{C^2M^2\|F\|_1}{\sigma^2\varepsilon^2}
	 \leq \frac{C^2M^2\frac{1}{\varepsilon^2}\left(\frac{\sigma^3\varepsilon^3}{CM^2}\right)}{\sigma^2\varepsilon^2} = C\sigma\varepsilon^{-1}.
\end{align*}
Thus,
\[
	|\{\omega \in [0,1]^2 \st \|f_\omega-g_\omega\|_2^2 \leq C^{-1}M^{-1}\sigma\varepsilon\}| \geq 1 - C\sigma\varepsilon^{-1},
\]
as required.
\end{proof}

The above proof also gives the following result.

\begin{corollary}
\label{sigalgcharfnc}
Let $0 < \varepsilon \ll 1$ and let $\chi$ be a character. Then the $\sigma$-algebra $\calb_{\varepsilon,\chi,\omega}$ described in the statement of Lemma \ref{sigalg} can be chosen to have the additional property that for every atom $\chi^{-1}(Q+\varepsilon\omega)$, $Q \in \bb{Q}_\varepsilon$, there exists a $K$-quasiperiodic function $g_{Q,\omega}$ that obeys $\|g_{Q,\omega}(\cdot)-1_Q(\chi(\cdot)-\varepsilon\omega)\|_2 < \sigma$ for every $\sigma > 0$, and in addition $\|g_{Q,\omega}\|_\infty \leq 2$.
\end{corollary}

We can concatenate the $\sigma$-algebras from Lemma \ref{sigalg}. If $\calb_1,\ldots,\calb_R$ are $\sigma$-algebras, denote by $\calb_1 \vee \cdots \vee \calb_R$ the smallest $\sigma$-algebra which contains all of them.

\begin{corollary}
\label{sigalgs}
Let $0 < \varepsilon_1,\ldots,\varepsilon_R \ll 1$ and let $\chi_1,\ldots,\chi_R$ be characters. Let $\calb_{\varepsilon_1,\chi_1},\ldots,\calb_{\varepsilon_R,\chi_R}$ be the $\sigma$-algebras arising from Lemma \ref{sigalg}. Then for every $\sigma > 0$, there exists $K = K(R,\sigma,\varepsilon_1,\ldots,\varepsilon_R)$ such that every function $f$ which is measurable with respect to $\calb_{\varepsilon_1,\chi_1} \vee \cdots \vee \calb_{\varepsilon_R,\chi_R}$ with $\|f\|_\infty \leq M$ is $(\sigma,K)$-almost periodic.
\end{corollary}

\begin{proof}
Since there are at most $C(R,\varepsilon_1,\ldots,\varepsilon_R)$ atoms in $\calb_{\varepsilon_1,\chi_1} \vee \cdots \vee \calb_{\varepsilon_R,\chi_R}$, it suffices to prove the claim in the case when $f$ is the indicator function of a single atom. Then $f$ is the product of $R$ indicator functions $f_1,\ldots,f_R$, where $f_j$ is the indicator function of an atom from $\calb_{\varepsilon_j,\chi_j}$. Let $g_j$ be a $K(\sigma/(R2^{R-1}),\varepsilon_j)$-quasiperiodic function approximating $f_j$ to within $\sigma/(R2^{R-1})$ as provided in Corollary \ref{sigalgs}; notice $\|g_j\|_\infty \leq 2$. Then $g = \prod g_j$ is a $K$-quasiperiodic function where $K = \prod K(\sigma/(R2^{R-1}),\varepsilon_j)$ depends only on $R,\sigma,\varepsilon_1,\ldots,\varepsilon_R$. Finally, by Lemma \ref{prodapprox} with $C = 2$, $p = 2$, and $\kappa = \sigma/(R2^{R-1})$, we have
\[
	\left\|\prod_{j=1}^R f_j - \prod_{j=1}^R g_j\right\|_2 \leq \sigma.
\]
\end{proof}

\subsection{Proof of Proposition 5.1}

We will need two more auxiliary results, analogous to \cite[Lemma 2.10 and 2.11]{tao}.

\begin{lemma}
\label{nonuniformitygivesstructure}
Let $b$ be a function bounded by $M$ with $\|\h{b}\|_\infty \geq
\sigma > 0$. Then there exists $0 < \varepsilon \ll \sigma$, a
character $\chi$, and an associated $\sigma$-algebra
$\calb_{\varepsilon,\chi}$ (as defined earlier in this section) such that
\[
	\|\bb{E}(b | \calb_{\varepsilon,\chi})\|_2 \geq C^{-1}\sigma.
\]
\end{lemma}
\begin{proof}
Since $\|\widehat{b}\|_{\infty} \geq \sigma$, there exists a character
$\chi$ such that
\begin{equation} \label{character}
\left| \int_{[0,1]^n} b(x) \chi(x) \, dx \right|
\geq \frac{\sigma}{2}.
\end{equation}
On the other hand, the $\sigma$-algebra  $\calb_{\varepsilon,\chi}$ is
generated by the atoms $\{ \chi^{-1}(Q + \varepsilon \omega); Q \in
\mathbb Q_{\varepsilon} \}$ for some $\omega$ in the unit square. On each
atom, $\chi$ can vary by at most $C \varepsilon$, hence
\begin{equation} \label{character-upper}
	\|\chi - \mathbb E(\chi| \calb_{\varepsilon,\chi})\|_{\infty} \leq C \varepsilon.
\end{equation}
Since $b$ is bounded by $M$, (\ref{character}) and (\ref{character-upper}) yield
\[ \int_{[0,1]^n} b(x) \mathbb E(\chi|\calb_{\varepsilon,\chi})(x) \,
dx \geq \frac{\sigma}{2} - MC\varepsilon.\]
Conditional expectation being self-adjoint, the inequality above
may be rewritten as
\[ \int_{[0,1]^n} \mathbb E(b|\calb_{\varepsilon,\chi})(x) \chi(x) \,
dx \geq \frac{\sigma}{2} - MC\varepsilon.\]
Recalling that $\chi$ is bounded above by 1, the desired result now follows
by choosing $\varepsilon$ sufficiently small relative to $\sigma$ and
$M$, and applying  Cauchy-Schwarz inequality to the integral on the
left.
\end{proof}

\begin{lemma}
\label{fncdecomp}
Let $F: \rr^+ \times \rr^+ \to \rr^+$ be an arbitrary function, let $0 < \delta \leq 1$, and let $f \geq 0$ be a function bounded by $M$ with $\int f \geq \delta$. Let $\sigma$ satisfy (\ref{e-sigma}). Then there exists a $K$ with $0 < K \leq C(F,\delta)$ and a decomposition $f = g+b$ where $g \geq 0$ is a bounded $(\sigma,K)$-almost periodic function with $\int g \geq \delta$, and $b$ obeys the bound
\begin{equation}
\label{bhatbound}
	\|\h{b}\|_\infty \leq F(\delta,K).
\end{equation}
\end{lemma}

The proof of Lemma \ref{fncdecomp} is exactly identical to \cite[2.11]{tao}.

\begin{proof}[Proof of Proposition 5.1]
Let $F: \rr^+ \times \rr^+ \to \rr^+$ be a function to be chosen later. Decompose $f = g+b$ as in Lemma \ref{fncdecomp}. By Lemma \ref{APrec},
\[
	\Lambda(g) \geq c(K,\delta,M).
\]
By Proposition \ref{decompapprox}, (\ref{bhatbound}), and the above inequality,
\[
	\Lambda(f) \geq c(K,\delta,M) + O(C(M,\delta)F(\delta,K)).
\]
By choosing $F$ sufficiently small and since $K \leq C(F,\delta)$, we get
\[
	\Lambda(f) \geq c(\delta,M)
\]
as required.
\end{proof}

\subsection{Quantitative Szemer\'edi bounds fail for general $\mathbb A$}
At the beginning of this section, we restricted to the case when $A_j$ is of the form $A_j = (I_{n \times n} \ B_j)$ where $B_j$ are $n \times (m-n)$ matrices. The reason for this is that generic $A_i$ will not provide a lower bound on $\Lambda$ when $\int f = \delta$, even when satisfying the non-degeneracy condition.

In the case $n = 2, k = 3, m = 4$, consider the function $f = \ind{B((0,1),\delta)}$, the indicator of the ball centered at $(0,1)$ with radius $\delta \leq 1/3$. Define
\begin{align*}
	A_1 & = \begin{pmatrix}
  	1 & 0 & 0 & 0 \\
  	0 & 1 & 0 & 0
 		\end{pmatrix}, \\
 	A_2 & = \begin{pmatrix}
  	0 & 0 & 1 & 0 \\
  	0 & 0 & 0 & 1
 		\end{pmatrix}, \\
 	A_3 & = \begin{pmatrix}
  	0 & 1 & 1 & 0 \\
  	1 & 0 & 0 & 1
 		\end{pmatrix}.
\end{align*}
It is clear that (\ref{gencond2alt}) holds for these matrices. In the integral defining $\Lambda$, we consider the conditions for $\vec{x} = (x_1,\ldots,x_4)$ to be in the support of $\prod_{i=1}^3 f(A_i\vec{x})$. The first term of the product gives $f(A_1\vec{x}) = f(x_1,x_2)$, and so in particular, $|x_2-1| < \delta$ which implies
\begin{equation}
\label{x2estimate}
	|x_2| > 1-\delta.
\end{equation}
Similarly, considering the second term yields in particular
\begin{equation}
\label{x3estimate}
	|x_3| < \delta,
\end{equation}
while the third term gives
\begin{equation}
\label{x2x3estimate}
	|x_2+x_3| < \delta.
\end{equation}
On the other hand, (\ref{x2estimate}) and (\ref{x3estimate}) give
\[
	|x_2+x_3| \geq |x_2|-|x_3| > 1-2\delta \geq \delta.
\]
Then the support of $\prod_{i=1}^4 f(A_i\vec{x})$ is empty, and $\Lambda = 0$.



\section{Proof of the main theorem}\label{ch:ProofOfMainResult}

The preceding section gave a quantitative lower bound on the $\Lambda$ quantity in the case of absolutely continuous measures with bounded density. This suggests the strategy of decomposing the measure $\mu$ as $\mu = \mu_1 + \mu_2$ where $\mu_1$ is absolutely continuous with bounded density, and $\mu_2$ gives negligible contribution. In light of the Fourier form of $\Lambda$, the key property of $\mu_2$ here will be having good bounds on the Fourier transform.

Let $\phi \in \cals(\rr^n)$ be a non-negative function supported on $B(0,1)$ with $\int \phi = 1$. For any positive integer $N$, define $\phi_N(x) = N^n\phi(Nx)$. Let $N \gg 1$ be a large constant to be determined later, and let
\[
	\mu_1(x) = \mu*\phi_N(x).
\]
Clearly, $\mu_1 \geq 0$ is a $C^{\infty}$ function of compact support with $\int d\mu_1 = 1$. Since $\phi_N$ is supported on $B(0,N^{-1})$,
\begin{align*}
	|\mu_1(x)| & \leq \int_{B(x,N^{-1})} |\phi_N(x-y)| \ d\mu(y) \\
			& = \int_{B(x,N^{-1})} N^n|\phi(N(x-y))| \ d\mu(y) \\
			& \leq CN^n \mu(B(x,N^{-1})) \\
			& \leq CN^{n-\alpha}
\end{align*}
where the last inequality follows by the ball condition (a). Then
$|\mu_1(x)| \leq M = Ce$ if $N = e^{1/(n-\alpha)}$, which tends to infinity as $\alpha \to n^-$.

Focusing now on $\mu_2$, we will prove that
\begin{equation} \label{muFourier}
\bigl| \widehat{\mu}_2(\xi)\bigr| \lesssim N^{-\frac{\varepsilon
    \beta}{2}} (1 + |\xi|)^{-\frac{\beta}{2}(1 - \varepsilon)}
\end{equation}
for some constant $\varepsilon > 0$ to be chosen later.
Since $\int \phi = 1$ and $\phi \in \cals(\rr^n)$,
\begin{align*}
	|1-\h{\phi}(\xi)| = |\h{\phi}(0)=\h{\phi}(\xi)| = \left|\int_0^1 \frac{d}{dt}\h{\phi}(t\xi) \ dt\right| = \int_0^1 |\xi \cdot \nabla\phi| \ dt \leq C|\xi|.
\end{align*}
In particular, defining $\mu_2 = \mu - \mu_1$ we have
\[
	|\h{\mu}_2(\xi)| \lesssim |\h{\mu}(\xi)| \min(1,|\xi|N^{-1}).
\]
Notice if $|\xi| \geq N$, then
\begin{align*}
	|\h{\mu}_2(\xi)| & \leq |\h{\mu}(\xi)| \\
			& \lesssim (1+|\xi|)^{-\beta/2} \\
			& = (1+|\xi|)^{-\varepsilon\beta/2}(1+|\xi|)^{-\beta/2(1-\varepsilon)} \\
			& \lesssim N^{-\varepsilon\beta/2}(1+|\xi|)^{-\beta/2(1-\varepsilon)}.
\end{align*}
On the other hand, if $|\xi| < N$, then we still have
\begin{align*}
	|\h{\mu}_2(\xi)| & \leq |\h{\mu}(\xi)| \\
			& \lesssim (1+|\xi|)^{-\beta/2}|\xi|N^{-1} \\
			& = (1+|\xi|)^{-\beta/2}|\xi|^{\varepsilon\beta/2}|\xi|^{1-\varepsilon\beta/2}N^{-1} \\
			& \lesssim N^{-\varepsilon\beta/2}(1+|\xi|)^{-\beta/2(1-\varepsilon)}.
\end{align*}

Now, decompose
\[
\Lambda^*(\widehat{\mu}) = \Lambda^*(\widehat{\mu_1}) + \Lambda(\widehat{\mu_2},\widehat{\mu_1},\ldots,\widehat{\mu_1}) + \cdots + \Lambda(\widehat{\mu_2}).
\]
By Proposition \ref{absctslowerbd}, $\Lambda^*(\widehat{\mu_1})=\Lambda(\mu_1) > c(\delta,M)$. It remains to show the $\Lambda^*$ quantities containing at least one copy of $\mu_2$ are negligible relative to $c(\delta,Ce)$. These quantities can be written as $\Lambda^*(g_1,\ldots,g_k)$ where for each $1 \leq j \leq k$, $g_j$ is either $\h{\mu_1}$ or $\h{\mu_2}$ and at least one $g_j$ is $\h{\mu_2}$. Without loss of generality, suppose $g_1 = \h{\mu}_2$, so that
\[
	|g_1(\eta_1)| \lesssim N^{-\varepsilon\beta/2}(1+|\eta_j|)^{-\beta/2(1-\varepsilon)}
\]
by the above estimate on $\h{\mu}_2$. For $j \geq 2$, we have
\[
	|g_j(\eta_j)| \lesssim (1+|\eta_j|)^{-\beta/2(1-\varepsilon)}
\]
by the above estimate on $\h{\mu}_2$ and the general Fourier decay condition (b) on $\mu$. Then
\begin{align*}
	\Lambda^*(g_1,\ldots,g_k) & = \int_S \prod_{j=1}^k g_j(\eta_j) \ d\sigma \\
			& \leq N^{-\varepsilon\beta/2} \int_S \prod_{j=1}^k(1+|\eta_j|)^{-\beta/2(1-\varepsilon)} \ d\sigma.
\end{align*}
Since $\beta > 2(nk-m)/k$, we may choose $\varepsilon > 0$ so that $\beta' = \beta(1-\varepsilon) > 2(nk-m)/k$. Then by Proposition \ref{lambdaextbound} with $\beta'$ in place of $\beta$, the integral above is bounded by a constant independent of $N$. Then we may choose $N$ sufficiently large that $\Lambda^*(g_1,\ldots,g_k) \leq 2^{-k}c(\delta,M)$, and so
\[
	\Lambda^*(\widehat{\mu}) \geq 2^{-k}c(\delta,M).
\]



\section{Examples}\label{ch:Examples}


For a fixed choice of $n \geq 1$ and $k \geq 3$, let $m = n\lceil (k+1)/2 \rceil$, the smallest value allowed by (\ref{nmkcond}). Non-degeneracy in this case will be the condition
\[
    \mathrm{rank}\begin{pmatrix}
  A_{i_1} \\
  \vdots \\
  A_{i_{m/n}} \end{pmatrix} = \mathrm{rank}\begin{pmatrix}
  I_{n \times n} & B_{i_1} \\
  \vdots & \vdots \\
  I_{n \times n} & B_{i_{m/n}} \end{pmatrix} = \lceil m,
\]
for $i_1,\ldots,i_{\lceil (k+1)/2 \rceil} \in \{1,\ldots,k\}$ distinct. Reducing,
\[
    \mathrm{rank}\begin{pmatrix}
  I_{n \times n} & B_{i_1} \\
	0_{n \times n} & B_{i_2}-B_{i_1} \\
	\vdots & \vdots \\
  0_{n \times n} & B_{m/n}-B_{i_1} \end{pmatrix} = m.
\]
Since $I_{n \times n}$ is of rank $n$, it suffices for
\begin{equation}
\label{reduceddegencond}
    \mathrm{rank}\begin{pmatrix}
  B_{i_2}-B_{i_1} \\
	\vdots \\
  B_{i_{m/n}}-B_{i_1} \end{pmatrix} = m-n,
\end{equation}
for $i_1,\ldots,i_{m/n} \in \{1,\ldots,k\}$ distinct. Notice that while it is necessary to check (\ref{reduceddegencond}) for every choice of $m/n$ indices $i_1,\ldots,i_{m/n}$, we do not need to check for permutations of the indices, any permutation suffices.

\begin{eg}[{\bf Triangles}]
\label{ex:triangles}
We now prove the claim in Corollary \ref{cor-triangles} that if $a,b,c$ are three distinct points in the plane, then any set $E\subset\rr^2$ obeying the assumptions of Theorem \ref{mainresult} with $\varepsilon_0$ small enough (depending on $C$ and on $a,b,c$) must contain a similar copy of the triangle $\triangle abc$. Note that our proof allows for degenerate triangles where $a,b,c$ are colinear.

Let $\theta$ be the angle between the line segments $\overline{ab}$ and $\overline{ac}$, measured counter-clockwise, and let $\lambda=\frac{|c-a|}{|b-a|}$. Permuting the points $a,b,c$ if necessary, we may assume without loss of generality that $\theta\in(0,\pi]$. Then it suffices to prove that $E$ contains a configuration of the form
\begin{equation}\label{isotriangles}
	x,\ \ x+y,\ \ x+\lambda y_\theta,
\end{equation}
where $y_\theta$ is the vector $y$ rotated by an angle $\theta$ counter-clockwise, for some $x,y\in\rr^2$ with $y\neq 0$.

Fix $n = 2$, $k = 3$, and $m = 4$. Let $B_1 = 0_{2 \times 2}$, By (\ref{reduceddegencond}), non-degeneracy means that
\[
    \mathrm{rank} (B_{j}) = 2, \hspace{1cm} \mathrm{rank}\begin{pmatrix}
  B_{3}-B_{2} \end{pmatrix} = 2,
\]
for $j = 2,3$.
With $\theta \in (0,\pi]$ and $\lambda>0$ as above, let
\begin{align*}
	B_2 = \begin{pmatrix}
  1 & 0 \\
  0 & 1 \end{pmatrix}, \hspace{1cm}
  B_3 = \begin{pmatrix}
  \lambda \cos\theta & -\lambda \sin\theta \\
 \lambda \sin\theta & \lambda \cos\theta \end{pmatrix}.
\end{align*}
It is easy to check that non-degeneracy holds, and this collection of matrices corresponds to configurations of the form (\ref{isotriangles}).
Letting $V=\{0\}$, Theorem \ref{mainresult} asserts that any set $E\subset \rr^2$ obeying its assumptions with $\varepsilon_0$ small enough must contain
such a configuration, non-degenerate in the sense that $y\neq 0$.
This proves Corollary \ref{cor-triangles}.

\end{eg}

\begin{eg}[{\bf Colinear triples}]
We prove that if $a,b,c$ are three distinct colinear points in $\rr^n$, then any set $E\subset\rr^n$ obeying the assumptions of Theorem \ref{mainresult} with $\varepsilon_0$ small enough (depending on $C$ and on $a,b,c$) must contain a non-degenerate similar copy of $\{a,b,c\}$.

Without loss of generality, suppose $|c-a| > |b-a|$ Let $\lambda=\frac{|c-a|}{|b-a|} > 1$. Then it suffices to prove that $E$ contains a configuration of the form
\begin{equation}\label{3linpts}
	x,\ \ x+y,\ \ x+\lambda y,
\end{equation}
for some $x,y\in\rr^n$ with $y\neq 0$.

Fix a positive integer $n$, $k = 3$, and $m = 2n$. Let $B_1 = 0_{n \times n}$, $B_2 = I_{n \times n}$,
$B_3 = \lambda I_{n \times n}$.
Similarly to Example \ref{ex:triangles}, this system of matrices produces configurations of the form (\ref{3linpts}), and the non-degeneracy condition (\ref{reduceddegencond}) becomes
\[
    \mathrm{rank} (B_{j}) = n, \hspace{1cm} \mathrm{rank}\begin{pmatrix}
  B_{3}-B_{2} \end{pmatrix} = n,
\]
for $j = 2,3$, which is easy to check for $B_j$ as above. Applying Theorem \ref{mainresult} with
$V=\{0\}$ as before, we get the desired conclusion.

\end{eg}

\begin{eg}[{\bf Parallelograms}]
\label{egparallelograms}
We now prove Corollary \ref{cor-parallelo}.
Fix $n \geq 1$, $k = 4$, and $m = 3n$. Let $B_1 = 0_{n \times 2n}$; (\ref{reduceddegencond}) tells us non-degeneracy will be the condition
\[
    \mathrm{rank}\begin{pmatrix}
  B_{i_1} \\
  B_{i_2} \end{pmatrix} = 2n, \hspace{1cm} \mathrm{rank}\begin{pmatrix}
  B_{2}-B_{4} \\
  B_{3}-B_{4} \end{pmatrix} = 2n
\]
for $i_1,i_2 \in \{2,3,4\}$ distinct.

Let
\begin{align*}
	B_2 & = \begin{pmatrix}
  I_{n \times n} \ 0_{n \times n} \end{pmatrix}, \\
  B_3 & = \begin{pmatrix}
  0_{n \times n} \ I_{n \times n} \end{pmatrix}, \\
  B_4 & = B_2+B_3 = \begin{pmatrix}
  I_{n \times n} \ I_{n \times n} \end{pmatrix}.
\end{align*}
Non-degeneracy clearly holds, and this collection of matrices corresponds to configurations of the form
\begin{equation}\label{parallelo}
x,\ \
x+\begin{pmatrix} y_1 \\ \vdots \\ y_n \end{pmatrix},\ \
x+\begin{pmatrix} y_{n+1} \\ \vdots \\ y_{2n} \end{pmatrix},\ \
x+\begin{pmatrix} y_1+y_{n+1} \\ \vdots \\ y_n+y_{2n} \end{pmatrix},
\end{equation}
for some $x\in\rr^n$ and $y_1,\ldots,y_{2n} \in \rr$. Geometrically, such configurations
describe 2-dimensional parallelograms. To exclude degenerate cases where the parallelogram becomes a line segment,
we define the ``exceptional" subspaces
$$
V_1=\{y\in\rr^{2n}:\ y_1=\dots=y_n=0\}, \
$$
$$
V_2=\{y\in\rr^{2n}:\ y_{n+1}=\dots=y_{2n}=0\}, \
$$
$$
V_3=\{y\in\rr^{2n}:\ y_1+y_{n+1}=0, \dots, y_n+y_{2n}=0\}, \
$$
$$
V_4=\{y\in\rr^{2n}:\ y_1-y_{n+1}=0, \dots, y_n-y_{2n}=0\}
$$
Then Theorem \ref{mainresult} provides for the existence of parallelograms with $y$ not in $V_1,V_2,V_3,V_4$, so that
the four points in (\ref{parallelo}) are all distinct.

\end{eg}

\begin{eg}[{\bf Polynomial configurations}]
Finally, we prove Corollary \ref{cor-poly}. We will in fact prove a stronger statement, namely that the result in Corollary \ref{cor-poly} holds in $\rr^n$ for all $n\geq 3$, with (\ref{poly1}) replaced by 4-point configurations defined below in Corollary \ref{cor-poly2}.

As in Example \ref{egparallelograms}, fix $n \geq 1$, $k = 4$, and $m = 3n$, and let $B_1 = 0_{n \times 2n}$. We will use a Vandermonde-style matrix for the remaining $B_i$. To make the notation less cumbersome, for a function
\begin{align*}
	g: \nn \times \nn & \to \rr \\
		(i,j) & \mapsto g(i,j),
\end{align*}
we denote by $(g(i,j))_{a \times b}$ the $a \times b$ matrix whose entry in the $i$th row and $j$th column is given by $g(i,j)$.

\begin{corollary}\label{cor-poly2}
Let $a_1, \ldots, a_{2n} > 1$ be distinct real numbers, and let $\eta,d \in \nn$. Consider the following matrices:
\[
B_2 = (a_j^{\eta+(i-1)d})_{n \times 2n}, \hspace{0.5cm} B_3 = (a_j^{\eta+(n+i-1)d})_{n \times 2n}, \hspace{0.5cm} B_4 = (a_j^{\eta+(2n+i-1)d})_{n \times 2n}.
\]
Suppose that $E\subset\rr^n$ obeys the assumptions of Theorem \ref{mainresult}, with $\epsilon_0$ small enough depending on $C$ and $a_i$. Then $E$ contains a configuration of the form
\begin{equation}\label{poly2}
x,\ \ x+B_2y,\ \ x+B_3y,\ \ x+B_4y
\end{equation}
for some $x\in \rr^n$ and $y\in\rr^{2n}$ with $B_iy\neq 0$ for $i=2,3,4$.
\end{corollary}

The proof of Corollary \ref{cor-poly2} will rely on two short lemmas.

\begin{lemma}
\label{numroots}
Suppose $0 \leq \eta_1 < \eta_2 < \ldots < \eta_t$ are integers. Then for any choice of constants $c_1,c_2,\ldots,c_t$ that are not all zero, the polynomial
\[
	P(x) = \sum_{i=1}^t c_ix^{\eta_i}
\]
has fewer than $t$ distinct positive roots.
\end{lemma}

\begin{proof}
We prove this with induction. For $t=1$, it is clear that $c_1x^{\eta_1}$ cannot have a positive root since $c_1 \neq 0$, so the base case is satisfied. We make the inductive hypothesis that the lemma holds for $t$, and check $t+1$. Suppose to the contrary that there exist constants $c_1,c_2,\ldots,c_{t+1}$, not all zero, such that the polynomial
\[
	P(x) = \sum_{i=1}^{t+1} c_ix^{\eta_i}
\]
has at least $t+1$ distinct positive roots. But then
\[
	x^{-\eta_1}P(x) = c_1+c_2x^{\eta_2-\eta_1}+\cdots+c_{t+1}x^{\eta_{t+1}-\eta_1},
\]
so by Rolle's Theorem, the following polynomial has at least $t$ distinct positive roots:
\begin{align*}
	P_1(x) & := \frac{d}{dx}(x^{-\eta_1}P(x)) \\
			& = c_2(\eta_2-\eta_1)x^{\eta_2-\eta_1-1}+c_3(\eta_3-\eta_1)x^{\eta_3-\eta_1-1}+\cdots+c_{t+1}(\eta_{t+1}-\eta_1)x^{\eta_{t+1}-\eta_1-1} \\
			& = \sum_{i=1}^{t}c_{i+1}(\eta_{i+1}-\eta_1)x^{\eta_{i+1}-\eta_1-1}
\end{align*}
Since $\eta_i$ were strictly increasing integers, $c_{i+1}(\eta_{i+1}-\eta_1)$ are not all zero, and $\eta_{i+1}-\eta_1-1 \geq 0$ are strictly increasing integers. This contradicts the induction hypothesis and completes the proof.
\end{proof}

\begin{lemma}
\label{fullrklem}
If
\[
	A = (a_j^{\eta_i})_{t \times s}
\]
where $a_1,a_2,\ldots,a_s$ are distinct, positive real numbers and $0 \leq \eta_1 < \eta_2 < \ldots < \eta_t$ are integers, then $A$ has full rank.
\end{lemma}

\begin{proof}
Without loss of generality, $t \leq s$. It suffices to show the following submatrix has full rank:
\[
	A_t = (a_j^{\eta_i})_{t \times t}.
\]
This holds if and only if $\det A_t \neq 0$. If to the contrary $\det A_t = 0$, then we can find constants $c_1,c_2,\ldots,c_t$ that are not all zero such that $\sum_i^t c_iR_i = \vec{0}$, where $R_i$ is the $i$th row of $A_t$; considering the $k$th position this says $\sum_{i=1}^t c_ia_j^{\eta_i} = 0$ for $1 \leq j \leq t$. That is, the polynomial
\[
	P(x) = \sum_{i=1}^t c_ix^{\eta_i}
\]
has at least the $t$ distinct positive roots $x = a_j$ for $1 \leq j \leq t$. This contradicts Lemma \ref{numroots}, so we must have $\det A_t \neq 0$ and hence $A$ has full rank.
\end{proof}

\begin{proof}[Proof of Corollary \ref{cor-poly2}]
By Lemma \ref{fullrklem},
\[
	\mathrm{rank}\begin{pmatrix}
  B_{i_1} \\
  B_{i_2} \end{pmatrix} = 2n
\]
for $i_1,i_2 \in \{2,3,4\}$ distinct. It remains to check
\begin{equation}
\label{vandermondematrixcheck}
	\mathrm{rank}\begin{pmatrix}
  B_{2}-B_{4} \\
  B_{3}-B_{4} \end{pmatrix} = \mathrm{rank}\begin{pmatrix}
  (a_j^{\eta+(i-1)d}-a_j^{\eta+(2n+i-1)d})_{n \times 2n} \\
  (a_j^{\eta+(n+i-1)d}-a_j^{\eta+(2n+i-1)d})_{n \times 2n} \end{pmatrix} = 2n.
\end{equation}
For constants $c_1,\ldots,c_{2n}$, consider the polynomial

\begin{align}
\begin{split}
	Q_{c_1,\ldots,c_n}(x) & = c_1(x^{\eta}-x^{\eta+2nd}) + c_2(x^{\eta+d}-x^{\eta+(2n+1)d}) + \cdots + c_n(x^{\eta+(n-1)d}-x^{\eta+(3n-1)d}) \\
					& \qquad \qquad + c_{n+1}(x^{\eta+nd}-x^{\eta+2nd}) + \cdots + c_{2n}(x^{\eta+(2n-1)d}-x^{\eta+(3n-1)d}) \label{testpoly}
	\end{split}
\end{align}

If (\ref{vandermondematrixcheck}) fails to hold, then as in the proof of Lemma \ref{fullrklem}, there are constants $c_1,\ldots,c_{2n}$ not all $0$ whose corresponding polynomial $Q(x) := Q_{c_1,\ldots,c_n}(x)$ has at least the $2n$ distinct roots $a_1,\ldots,a_{2n}$, all of which are larger than $1$. We may simplify $Q(x)$ as
\begin{align*}
	Q(x) & = c_1x^{\eta}(1-x^{2nd}) + c_2x^{\eta+d}(1-x^{2nd}) + \cdots + c_nx^{\eta+(n-1)d}(1-x^{2nd}) \\
					& \qquad \qquad + c_{n+1}x^{\eta+nd}(1-x^{nd}) + \cdots + c_{2n}x^{\eta+(2n-1)d}(1-x^{nd}) \\
			& = (1-x^{nd})[c_1x^{\eta}(1+x^{nd}) + c_2x^{\eta+d}(1+x^{nd}) + \cdots + c_nx^{\eta+(n-1)d}(1+x^{nd}) \\
					& \qquad \qquad + c_{n+1}x^{\eta+nd} + \cdots + c_{2n}x^{\eta+(2n-1)d}] \\
			& = (1-x^{nd})P(x),
\end{align*}
where
\[
	P(x) = c_1x^{\eta} + c_2x^{\eta+d} + \cdots + c_nx^{\eta+(n-1)d} + (c_1+c_{n+1})x^{\eta+nd} + \cdots + (c_n+c_{2n})x^{\eta+(2n-1)d}.
\]
The roots of $Q(x)$ which are larger than $1$ coincide with the roots of $P(x)$. Notice that not all of the coefficients of $P(x)$ are $0$, since not all of $c_1,\ldots,c_{2n}$ are $0$. Then by Lemma \ref{numroots}, $P(x)$ has fewer than $2n$ positive roots, a contradiction. Thus, (\ref{vandermondematrixcheck}) holds and $\{A_1,\ldots,A_k\}$ is non-degenerate. The result follows by applying Theorem \ref{mainresult}, with $V_i=\{y\in\rr^{2n}:\ B_iy=0\}$ for $i=2,3,4$.
\end{proof}

\end{eg}



\appendix

\section{Approximate Identity}\label{appendixapproxiden}

For a $p \times d$ ($p \leq d$) matrix $P$ of full rank $p$, define
\[
	V = \{\xi \in \rr^d \st P\xi = 0\} = \mathcal{N}(P),
\]
and $v = \mathrm{dim}(V) = d-p$.

\begin{definition}
\label{surfacemeas}
Fix an orthonormal basis $\{\vec{\alpha}_1,\ldots,\vec{\alpha}_v\}$ of $V$. The \emph{surface measure} $d\sigma$ on $V$ is defined as follows:
\[
	\int_V F\ d\sigma = \int_{\rr^v} F(x_1\vec{\alpha}_1+\cdots+x_v\vec{\alpha}_v)\ dx_1 \cdots dx_v
\]
for every $F \in C_c(\rr^d)$.
\end{definition}

Note that this definition is independent of the choice of basis. Indeed, if $\{\vec{\beta}_1,\ldots,\vec{\beta}_v\}$ is another orthonormal basis of $V$, then the mapping $(x_1,\ldots,x_v) \mapsto (y_1,\ldots,y_v)$ given by $\sum x_j\vec{\alpha}_j = \sum y_j\vec{\beta}_j$ is a linear isometry, hence given by an orthogonal matrix which has determinant $1$. Then $d\vec{x} = d\vec{y}$ and hence
\[
	\int_{\rr^v} F(x_1\vec{\alpha}_1+\cdots+x_v\vec{\alpha}_v)\ dx_1 \cdots dx_v = \int_{\rr^v} F(x_1\vec{\alpha}_1+\cdots+x_v\vec{\alpha}_v)\ dx_1 \cdots dx_v.
\]

\begin{lemma}
\label{approxidenlem}
For any $g \in C_c(\rr^d)$ and any $\Psi \in \cals(\rr^p)$ with $\Psi(0) \neq 0$,
\[
	\lim_{\varepsilon \to 0^+} \int_{\rr^d} g(y_1,y_2)\frac{1}{\varepsilon^{p}} \h{\Psi}\left(\frac{y_2}{\varepsilon}\right)\ dy_1dy_2 = \Psi(0)\int_{\rr^{d-p}} g(y_1,0)\ dy_1.
\]
Here, $y = (y_1,y_2) \in \rr^d$, with $y_1 \in \rr^{d-p}$ and $y_2 \in \rr^p$.
\end{lemma}

\begin{proof}
Fix $\kappa > 0$. Our goal is to show
\[
    \left|\int_{\rr^d} g(y_1,y_2)\frac{1}{\varepsilon^{p}} \h{\Psi}\left(\frac{y_2}{\varepsilon}\right)\ dy_1dy_2 - \Psi(0)\int_{\rr^{d-p}} g(y_1,0)\ dy_1\right| < \kappa
\]
for all $\varepsilon > 0$ sufficiently small. Then for every $\varepsilon > 0$, we have by definition of $\Psi(0)$
\begin{align*}
	& \left|\int_{\rr^d} g(y_1,y_2)\frac{1}{\varepsilon^{p}} \h{\Psi}\left(\frac{y_2}{\varepsilon}\right)\ dy_1dy_2 - \Psi(0)\int_{\rr^{d-p}} g(y_1,0)\ dy_1\right| \\
	= \ & \left|\int_{\rr^d} g(y_1,y_2)\frac{1}{\varepsilon^{p}} \h{\Psi}\left(\frac{y_2}{\varepsilon}\right)\ dy_1dy_2 - \int_{\rr^p}\int_{\rr^{d-p}} g(y_1,0)\frac{1}{\varepsilon^{p}}
\h{\Psi}\left(\frac{y_2}{\varepsilon}\right)\ dy_1dy_2\right| \\
	\leq \ & \iint |g(y_1,y_2)-g(y_1,0)|\frac{1}{\varepsilon^{p}} \left|\h{\Psi}\left(\frac{y_2}{\varepsilon}\right)\right|\ dy_1dy_2.
\end{align*}
We now partition our region of integration into where $|y_2|$ is small and where it is large. By uniform continuity of $g$ on its compact support $K$, we may choose $\eta > 0$ sufficiently small so that
\begin{equation}
	\sup_{y_1 \in K} |g(y_1,y_2)-g(y_1,0)| < \frac{\kappa}{2|\Psi(0)|}
\end{equation}
if $|y_2| \leq \eta$. On this region,
\begin{align*}
	& \iint_{|y_2| \leq \eta} |g(y_1,y_2)-g(y_1,0)|\frac{1}{\varepsilon^{p}} \left|\h{\Psi}\left(\frac{y_2}{\varepsilon}\right)\right|\ dy_1dy_2 \\
	< \ & \iint \frac{\kappa}{2|\Psi(0)|}\frac{1}{\varepsilon^{p}} \left|\h{\Psi}\left(\frac{y_2}{\varepsilon}\right)\right|\ dy_1dy_2 = \frac{\kappa}{2}.
\end{align*}

Since $\h{\Psi}$ is integrable, we can make the tail integral as small as we would like. In particular, there exists $\varepsilon > 0$ sufficiently small relative to $\eta$ so that
\begin{equation}
	\int_{|y_2| > \eta/\varepsilon} |\h{\Phi}(y_2)|\ dy_2 < \frac{\kappa}{4\|g\|_\infty \diam(K)}.
\end{equation}
Then for this $\varepsilon$,
\begin{align*}
	& \iint_{|y_2| > \eta, y \in K} |g(y_1,y_2)-g(y_1,0)|\frac{1}{\varepsilon^{p}} \left|\h{\Psi}\left(\frac{y_2}{\varepsilon}\right)\right|\ dy_1dy_2 \\
	\leq \ & 2\|g\|_\infty \diam(K) \int_{|y_2| > \eta/\varepsilon} |\h{\Psi}(y_2)|\ dy_2 < \frac{\kappa}{2}.
\end{align*}
As this inequality holds for every $\kappa > 0$, the result follows.
\end{proof}

\begin{proposition}
\label{approxidenprop}
For any $P$ as above, there exists a constant $C_P > 0$ with the property that for any $\Phi \in \cals(\rr^p)$ with $\Phi(0) = 1$, the limit
\[
	\lim_{\varepsilon \to 0^+} \int_{\rr^d} F(\xi)\frac{1}{\varepsilon^{p}}\h{\Phi}\left(\frac{P\xi}{\varepsilon}\right)\ d\xi
\]
exists and equals $C_P\int_V F\ d\sigma$.
\end{proposition}

\begin{proof}
Fix $\Phi \in \cals(\rr^p)$ with $\Phi(0) = 1$. Let $\{\vec{\alpha}_1,\ldots,\vec{\alpha}_v\}$ be an orthonormal basis of $V$. Extend this to an orthonormal basis of $\rr^d$, say $\{\vec{\alpha}_1,\ldots,\vec{\alpha}_v,\vec{\alpha}_{v+1},\ldots,\vec{\alpha}_d\}$. Given any function $F \in C_c(\rr^d)$, we define $G_F: \rr^d \to \rr$ as follows:
\[
	G(x_1,\ldots,x_d) = F\left(\sum_{j=1}^d x_j\vec{\alpha}_j\right).
\]
Notice $G_F \in C_c(\rr^d)$ as well. Then by definition,
\begin{align*}
	\int_V F\ d\sigma & = \int_{\rr^v} F\left(\sum_{j=1}^v x_j\vec{\alpha}_j\right)\ dx_1 \cdots dx_v \\
			& = \int_{\rr^v} G_F(x_1,\ldots,x_v,0,\ldots,0)\ dx_1 \cdots dx_v \\
			& = \frac{1}{\Psi(0)} \lim_{\varepsilon \to 0^+} \int_{\rr^d} G_F(x_1,\ldots,x_v,x_{v+1},\ldots,x_d)\frac{1}{\varepsilon^{p}} \h{\Psi}\left(\frac{x_{d+1}}{\varepsilon},\ldots,\frac{x_d}{\varepsilon}\right)\ d\vec{x}
\end{align*}
for any $\Psi \in \cals(\rr^d)$ with $\Psi(0) \neq 0$, by Lemma \ref{approxidenlem}. By the definition of $G_F$, this gives
\begin{equation}
\label{propeqn1}
	\int_V F\ d\sigma = \frac{1}{\Psi(0)} \lim_{\varepsilon \to 0^+} \int_{\rr^d} F(\xi) \frac{1}{\varepsilon^{p}} \h{\Psi}\left(\frac{x_{v+1}}{\varepsilon},\ldots,\frac{x_d}{\varepsilon}\right)\ d\vec{x},
\end{equation}
for any $\Psi \in \cals(\rr^d)$ with $\Psi(0) \neq 0$, where we denote $\xi = \left(\sum_{j=1}^d x_j\vec{\alpha}_d\right)$. Now, let $Q$ be the $p \times p$ matrix defined by
\[
	Q\begin{pmatrix} x_{v+1} \\ \vdots \\ x_d \end{pmatrix} = \sum_{j=v+1}^d x_jP\vec{\alpha}_j.
\]
Since $P$ is of full rank and acting on basis vectors, $Q$ is non-singular. Recall $V = \{\xi \st P\xi = 0\}$ and $\{\vec{\alpha}_1,\ldots,\vec{\alpha}_v\}$ is a basis for $V$, so $P\vec{\alpha}_j = 0$ for $1 \leq j \leq v$. Then
\begin{equation}
\label{propeqn2}
	\begin{pmatrix} x_{v+1} \\ \vdots \\ x_d \end{pmatrix} = Q^{-1}\sum_{j=v+1}^d x_jP\vec{\alpha}_j = Q^{-1}\sum_{j=1}^d x_jP\vec{\alpha}_j = Q^{-1}P\xi.
\end{equation}
Define $\Psi \in \cals(\rr^d)$ by
\[
	\h{\Psi}(\xi) = \h{\Phi}(Q\xi).
\]
Then by (\ref{propeqn1}) and (\ref{propeqn2}),
\begin{align*}
	\int_V F\ d\sigma & = \frac{1}{\Psi(0)} \lim_{\varepsilon \to 0^+} \int_{\rr^d} F(\xi) \frac{1}{\varepsilon^{p}} \h{\Psi}\left(\frac{Q^{-1}P\xi}{\varepsilon}\right)\ d\vec{x} \\
			& = \frac{1}{\Psi(0)} \lim_{\varepsilon \to 0^+} \int_{\rr^d} F(\xi) \frac{1}{\varepsilon^{p}} \h{\Phi}\left(\frac{P\xi}{\varepsilon}\right)\ d\vec{x}.
\end{align*}
Finally,
\begin{align*}
	\Psi(0) & = \int_{\rr^p} \h{\Psi}(\xi)\ d\xi \\
			& = \int_{\rr^p} \h{\Phi}(Q\xi)\ d\xi \\
			& = \frac{1}{|Q|}\int_{\rr^p} \h{\Phi}(\xi)\ d\xi \\
			& = \frac{1}{|Q|}\Phi(0) = \frac{1}{|Q|}
\end{align*}
and so the result follows with $C_P = |Q|$. Note that $Q$ is also independent of the choice of basis, and is a function only of $P$.
\end{proof}



\section{Supplementary facts from Section 5}\label{appendixsec5}

\begin{lemma}
\label{simapprox}
Given $0 < \varepsilon < 1$ and any integer $K \geq 1$, there exists a positive
constant $c'(\varepsilon, K)$ such that
\[
	|\{t \in [0,1] \st \|tv_\ell\| \leq \varepsilon \text{ for all } 1 \leq \ell \leq K\}| \geq c'(\varepsilon,K),
\]
for any choice of $v_1,\ldots,v_K \in \zz$.
\end{lemma}

\begin{proof}
Clearly it suffices to prove the lemma for the case when $\varepsilon \leq 1$. Let $N$ be the unique integer that $N^{-1} < \varepsilon \leq (N-1)^{-1}$, and consider the partition of the unit cube $[0,1]^K$ into $N^K$ disjoint cubes of side length $N^{-1}$. That is, the vertices of the cubes are at points of the form $N^{-1}\zz^K \ \mathrm{mod} \ 1$. Define
\[
	X_Q = \{t \in [0,1] \st (tv_1,\ldots,tv_K) \ \mathrm{mod} \ 1 \in Q\}.
\]
Then since $|[0,1]| = 1$, there must exist a cube $Q$ such that
\[
	|X_Q| \geq \frac{1}{N^K} \geq \left(\frac{\varepsilon}{2}\right)^K,
\]
where we have used that $\varepsilon \leq (n-1)^{-1}$. For $t \in X_Q$,
\[
	(tv_1,\ldots,tv_K) \ \mathrm{mod} \ 1 \in Q - Q \subseteq [-n^{-1},n^{-1}]^K \subseteq [-\varepsilon,\varepsilon]^K,
\]
and so $\|tv_\ell\| \leq \varepsilon$ for every $1 \leq \ell \leq K$. Notice
\[
	|X_Q-X_Q| \geq |X_Q| \geq \left(\frac{\varepsilon}{2}\right)^K,
\]
and so by symmetry
\[
	|\{t \in [0,1] \st \|tv_\ell\| \leq \varepsilon \text{ for all } 1 \leq \ell \leq K\}| \geq \frac{1}{2}\left(\frac{\varepsilon}{2}\right)^K.
\]
\end{proof}

\begin{corollary}
\label{cepsmeas}
Given $0 < \varepsilon < 1$, and integers $k, K, m, n \in \mathbb N$, $m
> n$, there exists a positive constant $c$ depending on all of these
quantities, for which the set
\[
	C_\varepsilon = \{y \in \rr^{m-n} \st \|A_j^tv_\ell \cdot y\| \leq \varepsilon, \text{ for all } 1 \leq j \leq k, 1 \leq \ell \leq K\}.
\]
defined as in Lemma \ref{APrec} obeys the size estimate
\[
	|C_\varepsilon| \geq c
\]
for any choice of matrices $\{ A_j \}$ and vectors $\{ v_{\ell}\}$.
\end{corollary}

\begin{proof}
Let $A_j^tv_\ell(i)$ denote the $i$th component of $A_j^tv_\ell$. Let
\[
	D_i = \{y_i \in [0,1] \st \|A_j^tv_\ell(i)y_i\| \leq \varepsilon/(m-n) \text{ for all } 1 \leq j \leq k, 1 \leq \ell \leq K\}.
\]
By Lemma \ref{simapprox},
\[
	|D_i| \geq c'(\varepsilon/(m-n),K)^k
\]
for $1 \leq i \leq m-n$. If $y = (y_1,\ldots,y_{m-n}) \in \prod_{i=1}^{m-n} D_i$, then
\[
	\|A_j^tv_\ell \cdot y\| = \left\|\sum_{i=1}^{m-n} A_j^tv_\ell(i)y_i\right\| \leq (m-n)\frac{\varepsilon}{m-n} = \varepsilon,
\]
and so
\[
	C_\varepsilon \supseteq D_1 \times \cdots D_{m-n}.
\]
Therefore,
\[
	|C_\varepsilon| \geq c'(\varepsilon/(m-n),K)^{k(m-n)} = c(\varepsilon,K).
\]
\end{proof}

\begin{lemma}
\label{prodapprox}
Suppose $\|f_j\|_\infty, \|g_j\|_\infty \leq C$ and $\|f_j-g_j\|_p \leq \kappa$ for $1 \leq j \leq R$, for some $1 \leq p \leq \infty$. Then
\[
	\left\|\prod_{j=1}^R f_j - \prod_{j=1}^R g_j\right\|_p \leq RC^{R-1}\kappa.
\]
\end{lemma}

\begin{proof}
\begin{align*}
	\left\|\prod_{j=1}^R f_j - \prod_{j=1}^R g_j\right\|_p & \leq \left\|f_1 \cdot \prod_{j=2}^{R} f_j - g_1 \cdot \prod_{j=2}^{R} f_j\right\|_p \\
					& \qquad + \left\|g_1 \cdot f_2 \cdot \prod_{j=3}^{R} f_j - g_1 \cdot g_2 \cdot \prod_{j=3}^{R} f_j\right\|_p \\
					& \qquad + \cdots + \left\|\prod_{j=1}^{R-1} g_j \cdot f_R - \prod_{j=1}^{R-1} g_j \cdot g_R\right\|_p \\
			& \leq \left(\prod_{j=2}^{R}\|f_j\|_\infty\right)\|f_1-g_1\|_p \\
					& \qquad + \|g_1\|_\infty\left(\prod_{j=3}^{R}\|f_j\|_\infty\right)\|f_2-g_2\|_p \\
					& \qquad + \cdots + \left(\prod_{j=1}^{R-1}\|g_j\|_\infty\right)\|f_R-g_R\|_p \\
			& \leq RC^{R-1}\kappa.
\end{align*}
\end{proof}

\begin{lemma}
\label{finitelymanyatoms}
For a fixed $\varepsilon, \chi, \omega$, define $\calb_{\varepsilon,\chi,\omega}$ as in Lemma \ref{sigalg}. Then $\calb_{\varepsilon,\chi,\omega}$ has at most $\pi 4\sqrt{2}/\varepsilon$ atoms.
\end{lemma}

\begin{proof}
Let $N$ be the number of atoms of $\calb_{\varepsilon,\chi,\omega}$. Recall that the atoms are of the form $\chi^{-1}(Q + \varepsilon \omega)$ for $Q \in \bb{Q}_\varepsilon$. Since the image of $\chi$ lies in the unit circle $\bb{S}^1$, the preimage of $(Q + \varepsilon \omega)$ under $\chi$ is non-empty if and only if $(Q + \varepsilon \omega) \cap \bb{S}^1 \neq \varnothing$. Since $\mathrm{diam}(Q + \varepsilon \omega) = \sqrt{2} \varepsilon$, this holds if and only if $(Q + \varepsilon \omega)$ lies in the $(\sqrt{2} \varepsilon)$-thickened unit circle,
\[
	\bb{S}^1_{\sqrt{2} \varepsilon} = \{x \in \cc \st \dist(x,\bb{S}^1) \leq \sqrt{2}\varepsilon\}.
\]
It is easy to calculate $|\bb{S}^1_{\sqrt{2} \varepsilon}| = \pi 4\sqrt{2}\varepsilon$ and $|Q + \varepsilon \omega| = \varepsilon^2$ for every $Q \in \bb{Q}_\varepsilon$, so since the squares $Q + \varepsilon \omega$ are disjoint,
\[
	N \leq \frac{|\bb{S}^1_{\sqrt{2} \varepsilon}|}{|Q + \varepsilon \omega|} = \frac{\pi 4\sqrt{2}}{\varepsilon}
\]
as claimed.
\end{proof}

\bibliographystyle{plain}

\vskip0.5in

\noindent Vincent Chan
\\ University of British Columbia \\ 1984 Mathematics Road \\ Vancouver BC, Canada V6T 1Z2 \\ {\em{Email: vchan@math.ubc.ca}}
\vskip0.2in
\noindent Izabella {\L}aba \\ University of British Columbia \\ 1984 Mathematics Road \\ Vancouver BC, Canada V6T 1Z2 \\ {\em{Email: ilaba@math.ubc.ca}}
\vskip0.2in
\noindent Malabika Pramanik \\ University of British Columbia, Vancouver \\ 1984 Mathematics Road \\ Vancouver BC, Canada V6T 1Z2 \\{\em{Email: malabika@math.ubc.ca}}

\end{document}